\documentclass[11pt]{article}

\usepackage{geometry}
\usepackage[utf8]{inputenc}
\usepackage[english]{babel}
\usepackage{lmodern}
\usepackage{microtype}

\usepackage{graphicx}
\usepackage{hyperref}
\usepackage{xcolor}
\usepackage{standalone}
\usepackage{tikz}
\usepackage{amsmath,amsfonts,amssymb}
\usepackage{amsthm}
\usepackage[capitalise]{cleveref}
\usepackage[shortlabels]{enumitem}
\usepackage[nocompress]{cite}

\usepackage[numbers]{natbib}

\usepackage{tabularx}

\makeatletter
\g@addto@macro\normalsize{%
  \setlength\abovedisplayskip{5pt}
  \setlength\belowdisplayskip{5pt}
  \setlength\abovedisplayshortskip{3pt}
  \setlength\belowdisplayshortskip{3pt}
}
\makeatother

\newcommand{\N}{\ensuremath{\mathbb{N}}}
\newcommand{\Z}{\ensuremath{\mathbb{Z}}}
\newcommand{\Q}{\ensuremath{\mathbb{Q}}}
\newcommand{\R}{\ensuremath{\mathbb{R}}}
\newcommand{\F}{\ensuremath{{\mathcal{F}}}}

\newcommand{\A}{\ensuremath{\mathcal{A}}}
\newcommand{\mS}{\ensuremath{\mathcal{S}}}

\newcommand{\floor}[1]{\lfloor #1 \rfloor}

\newcommand\mydots{\makebox[0.9em][c]{.\hfil.\hfil.}}

\renewcommand{\epsilon}{\varepsilon}

\theoremstyle{plain}
\newtheorem{theorem}{Theorem}[section]
\newtheorem{lemma}[theorem]{Lemma}
\newtheorem{corollary}[theorem]{Corollary}

\newtheorem{prop}[theorem]{Proposition}
\crefname{prop}{Proposition}{Propositions}

\theoremstyle{definition}
\newtheorem{definition}{Definition}[section]
\theoremstyle{remark}
\newtheorem{remark}[theorem]{Remark}

\title{Weakly and Strongly Aperiodic Subshifts of Finite Type on Baumslag-Solitar Groups}

\usepackage{fancyhdr}
\makeatletter
\renewcommand\@date{{%
  \vspace{-\baselineskip}%
  \vspace{-\baselineskip}%
  \vspace{-\baselineskip}%
  \large\centering
  \begin{tabular}{@{}c@{}}
    ~~~~Solène J. Esnay\textsuperscript{1,2} \\
    \normalsize julien.esnay@ens-lyon.fr
  \end{tabular}%
  \quad and\quad
  \begin{tabular}{@{}c@{}}
    ~~~~Etienne Moutot\textsuperscript{3}\\
    \normalsize etienne.moutot@math.cnrs.fr
  \end{tabular}

  \bigskip

  \normalsize

  \textsuperscript{1} Université de Lyon, CNRS, ENS de Lyon, UCBL, LIP, F-69342, Lyon~Cedex~07, France\par
  \textsuperscript{2} Institut de Math\'ematiques de Toulouse, Universit\'e Paul Sabatier, 118 route de Narbonne, F-31062 Toulouse~Cedex~9, France\par
  \textsuperscript{3} CNRS, Aix Marseille Univ., I2M, Marseille, France

  \bigskip
}}
\makeatother

\makeatletter
\def\keywords{\xdef\@thefnmark{}\@footnotetext}
\makeatother

\begin{document}

\maketitle

\begin{abstract}
We study the periodicity of subshifts of finite type (SFT) on Baumslag-Solitar groups.
We show that for residually finite Baumslag-Solitar groups there exist both strongly and weakly-but-not-strongly aperiodic SFTs. In particular, this shows that unlike $\Z^2$, but like $\Z^3$, strong and weak aperiodic SFTs are different classes of SFTs in residually finite BS groups. More precisely, we prove that a weakly aperiodic SFT on BS(m,n) due to Aubrun and Kari is, in fact, strongly aperiodic on BS(1,n); and weakly but not strongly aperiodic on any other BS(m,n).
In addition, we exhibit an SFT which is weakly but not strongly aperiodic on BS(1,n); and we show that there exists a strongly aperiodic SFT on BS(n,n).
\end{abstract}

\section*{Introduction}

The use of tilings as a computational model was initiated by Wang in the 60s~\cite{Wang1961} as a tool to study specific classes of logical formulas.
His model consists in square tiles with colors on each side, that can be placed next to each other if the colors match. 
Wang studied the tilings of the discrete infinite plane~($\Z^2$) with these tiles, now called Wang tiles; a similar model exists to tile the infinite line~($\Z$) with two-sided dominoes.
Wang realized that a key property of these tilings was the notion of periodicity.
At the time, it was already known that if a set of dominoes tiled~$\Z$, it was always possible to do it in a periodic fashion, and Wang suspected that it was the same for Wang tiles on~$\Z^2$.
However, a few years later, one of his students, Berger, proved otherwise by providing a set of Wang tiles that tiled the plane but only aperiodically~\cite{BergerPhD}. Numerous aperiodic sets of Wang tiles have been provided by many since then (see for example~\cite{Robinson1971, Kari2008, JeandelRao}).

The model of Wang tiles is actually equivalent to tilings using an arbitrary finite alphabet with adjacency rules, and it has become a part of symbolic dynamics, a more general way to encode a smooth dynamical system into symbolic states and trajectories. 
This approach by discretization (see~\cite{genesis} for a comprehensive historiography) is itself part of the field of discrete dynamical systems.
In this broader context, it is interesting to study the set of all possible tilings for a given finite list of symbols and adjacency rules, called a Subshift of Finite Type~(SFT).
% Such a set can be described by the set of Wang tiles producing it, or in a larger setting by a finite alphabet associated with a finite set of forbidden patterns.
$\Z$-SFTs have been studied extensively, and many of their properties are known (see for example~\cite{LindMarcus}).
{$\Z^d$-SFTs} (for $d\geq 2$) seem to be much more complex, as many results from the unidimensional case cannot be directly transcribed to higher dimensions.
In recent years, the even wider class of SFTs built over Cayley graphs of groups has been attracting more and more attention~\cite{commensurable, computable, Aubrun2018}.

In this broader setting, many questions about periodicity are still open~\cite{commensurable}.
For example there are many groups for which we do not know if there exists a non-empty SFT with only aperiodic configurations (called an aperiodic SFT).
Another relevant question is the relation between two notions of aperiodicity, weak and strong aperiodicity: the first one requires all configurations to have infinitely many distinct translates; the second one that no configuration has any period whatsoever. 
They are equivalent for SFTs over $\Z$ or $\Z^2$, but this is not the case for other groups (notably $\Z^3$), and for many the relation is not even known.

Baumslag-Solitar groups of parameters $m$ and $n$, commonly denoted by~$BS(m,n)$, initially gathered interest in symbolic dynamics because of their simple description and rich properties. 
Notably, their domino problem is undecidable~\cite{BS}; and Aubrun and Kari built a weakly aperiodic SFT in order to prove it.
Their proof is essentially sketched in the general case, where it encodes piecewise affine maps, as it is done on~$\Z^2$ in~\cite{Kari2008}. 
Only the $BS(2,3)$ case is detailed in their paper; and an explicit period is provided for a given configuration, which shows that the resulting SFT, although weakly aperiodic, is not strongly aperiodic.

In the present paper, after a few definitions~(\cref{sec:def}); 
we thoroughly reintroduce Aubrun and Kari's construction in the general case, 
and provide a precise proof of its weak aperiodicity by encoding piecewise linear maps~(\cref{sec:AK}) (although Aubrun and Kari sketched the proof for encoding piecewise affine maps, only piecwise linear ones are needed for the aperiodicity result). 
Then, we show that the resulting SFT is weakly but not strongly aperiodic on any~$BS(m,n)$~(\cref{subsec:weakmn}) (as sketched in \cite{BS}); except on $BS(1,n)$ where, with some extra work, we prove that it is strongly aperiodic~(\cref{subsec:strong1n}).
Then, using a different technique based on substitutions on words, we exhibit a weakly but not strongly aperiodic SFT on~$BS(1,n)$~(\cref{sec:weak1n}). 
Finally, by using tools from group theory and a theorem by Jeandel~\cite{computable}, we build a strongly aperiodic SFT on any~$BS(n,n)$~(\cref{sec:strongnn}).
These new results are summarized in the following table (in bold):

\begin{table}[h!]
  \begin{center}
    \resizebox{\columnwidth}{!}{
      \begin{tabular}{|c|c|c|}
        \hline
        Group & Strongly aperiodic SFT & Weakly-not-strongly aperiodic SFT\\
        \hline
        $\Z^2$
        &
        Yes (Berger \cite{BergerPhD})
        &
        No (Folklore)
        \\
        \hline
        $BS(1,n)$
        &
        \textbf{Yes, adapted from Aubrun-Kari} (\cref{subsec:strong1n})
        &
        \textbf{Yes, using substitutions} (\cref{sec:weak1n})
        \\
        \hline
        $BS(n,n)$
        &
        \textbf{Yes, using a theorem by Jeandel} (\cref{sec:strongnn})
        &
        Yes (\cref{subsec:weakmn}, Aubrun-Kari \cite{BS})
        \\
        \hline
        $BS(m,n)$
        &
        ?
        &
        Yes (\cref{subsec:weakmn}, Aubrun-Kari \cite{BS})
        \\
        \hline
      \end{tabular}
    }
  \end{center}
\end{table}

The end result is that any residually finite Baumslag-Solitar group $BS(m,n)$ with $|m|\geq 2$ or $|n|\geq 2$ has a strongly aperiodic SFT and a weakly but not strongly aperiodic SFT. The remaining question is whether a general~$BS(m,n)$ admits a strongly aperiodic SFT, when not in the $|m|=|n|$, $|m|=1$ or $|n|=1$ cases.

%%%%%%%%%%%%%%%%%%%%%%%%%%%%%%%%%%%%%%%%%%%%%

\section{Definitions}
\label{sec:def}

\subsection{Baumslag-Solitar groups}

A group $G$ is said to be \emph{finitely generated} if it has a presentation $G=\langle S \mid R \rangle$ with $S$ finite.
Given a presentation $\langle S \mid R \rangle$ of a group $G$, its (right) \emph{Cayley graph} is the graph $\Gamma_G = (G,E_\Gamma)$ whose vertices are the elements of $G$ and the edges are of the form $(g, gs)$ with $g\in G$ and $s\in S\cup S^{-1}$ a generator of $G$ or its inverse.

A group $G$ is said to be \emph{residually finite} if for any $g \in G$ that is not the identity, there is a normal subgroup $N \vartriangleleft G$ of finite index such that $g \notin N$.

The groups we are interested in in this paper are the \emph{Baumslag-Solitar groups}. They are defined, using two \emph{nonzero} integers $m,n$ as parameters, by the presentation
\[ BS(m,n) = \langle a,b \mid b a^{m} b^{-1} = a^n \rangle . \]

\begin{prop}[Meskin \cite{residuallyfinite}]
  \label{prop:residually}
  $BS(m,n)$ is residually finite $\Leftrightarrow$ $|m|=1$ or $|n|=1$ or $|m| = |n|$.
\end{prop}

\subsection{Subshifts and tilings}
\label{subsec:tilings}
In a general sense, a \emph{dynamical system} is a set $S$ endowed with a topology $U$ and on which a group acts; if $\Z$ acts on it by the iteration of a function $f$, we may write $(S,U,f)$ instead of $(S,U,\Z)$.
Two dynamical systems $(S,U,G)$ and $(T,V,G)$ are \emph{conjugate} if there exists a continuous bijection $\phi\colon S \rightarrow T$ so that for any $s \in S$ and any $g \in G$, $\phi(g \cdot s) = g \cdot \phi(s)$ (where $G$ acts, respectively, on $S$ and on $T$).

Let $\A$ be a finite alphabet.
Let $G$ be a finitely generated group with presentation $\langle S \mid R \rangle$ and neutral element $e$. Let $x \in \A^G$ and $g,h \in G$: $G$ naturally acts on the left on $\A^G$ by $(g \cdot x)_h = x_{g^{-1}h}$.
The set $\A^G$, when endowed with the product topology $t$ and this action, forms a compact dynamical system $(\A^G,t,\cdot)$ called the \emph{full shift} over $G$. We call $x \in \A^G$ a \emph{configuration}.

A \emph{pattern} $p$ is a finite configuration $p \in \mathcal{A}^{P_p}$ where $P_p \subset G$ is finite.
We say that a pattern $p \in \mathcal{A}^{P_p}$ \emph{appears} in a configuration $x\in\A^{G}$ --~or that $x$ \emph{contains} $p$~-- if there exists $g \in G$ such that for every $h \in P_p$, $(g \cdot x)_{h} = p_{h}$, in this case we write $p\sqsubset x$.

The \emph{subshift} associated to a set of patterns \F, called set of \emph{forbidden patterns}, is defined by
\[
X_\F = \{ x \in \mathcal{A}^{G} \mid \forall p \in \mathcal{F}, p \not\sqsubset x \}
\]
that is, $X_\F$ is the set of all configurations that do not contain any pattern from~$\F$. Note that there can be several sets of forbidden patterns defining the same subshift $X$. A subshift can equivalently be defined as a closed set under both the topology and the $G$-action.\\
If $X=X_\F$ with $\F$ finite, then $X$ is called a \emph{Subshift of Finite Type}, SFT for short.

Let $X$ be a subshift on a group $G$ and $x \in X$.
The \emph{orbit} of $x$ is ${Orb_G(x) = \{g\cdot x \mid g\in G\}}$ and its \emph{stabilizer} $Stab_G(x) = \{ g\in G \mid g\cdot x = x \}$.
We say that $x$ is a \emph{strongly periodic configuration} if $|Orb_G(x)| < +\infty$, and that $x$ is a \emph{weakly periodic configuration} if $Stab_G(x) \neq \{e\}$.
If no configuration in $X$ is strongly periodic and the subshift is non-empty, then the subshift is said to be \emph{weakly aperiodic}.
If no configuration in $X$ is weakly periodic and the subshift is non-empty, then the subshift is said to be \emph{strongly aperiodic}.

A particular class of SFTs is obtained by considering Wang tiles over the Cayley graph of the group.
This can be done for any finitely generated group, and it turns out that any SFT can be encoded into an equivalent set of Wang tiles.
In order to make the definitions shorter, we limit ourselves to the definition of Wang tiles over $BS(m,n)$.
A \emph{Wang tileset} is a particular SFT where the alphabet is a set of \emph{Wang tiles} $\tau$, which are tuples of colors of the form $s = (t_1^s,\dots,t_m^s,l^s,r^s,b_1^s,\dots,b_n^s)$.
To make notations easier, we denote:
\begin{align*}
  s(\text{top}_1)     & = t_1^s \\
              & ~~\vdots \\
  s(\text{top}_m)   & = t_{m}^s \\
  s(\text{left})    & = l^s \\
  s(\text{right})   & = r^s \\
  s(\text{bottom}_1)  & = b_1^s \\
              & ~~\vdots \\
  s(\text{bottom}_n)  & = b_{n}^s
\end{align*}

\begin{figure}[ht]
  \centering
  \includegraphics[scale=0.9]{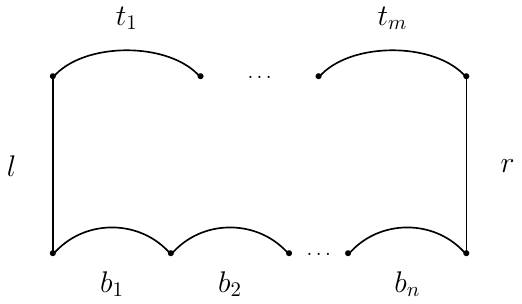}
  \caption{A Wang tile of $BS(m,n)$}
  \label{fig:tiles_BSmn}
\end{figure}

A \emph{tiling} is then a configuration $T\in \tau^{BS(m,n)}$ over the group using the alphabet $\tau$.
We say that a tiling is \emph{valid} if the colors of neighboring tiles match.
% %
That is, for any $g\in BS(m,n)$ and $T_g$ the associated tile at position $g$, we must have:
\begin{align*}
    T_g(\text{right})     & = T_{ga^m}(\text{left}) \\
    T_g(\text{top}_k)     & = T_{ga^{k-l}b}(\text{bottom}_l)
\end{align*}
for any $k \in \{1,\dots,m\}$ and $l \in \{1,\dots,n\}$.
\begin{figure}[ht]
  \centering
  \includegraphics[scale=0.9]{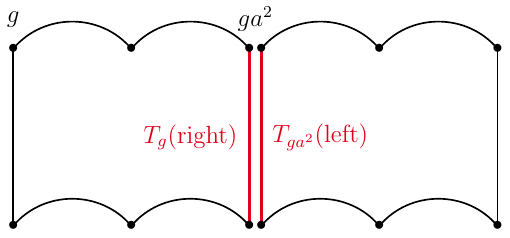}
    \hspace{0.3em}
  \includegraphics[scale=0.9]{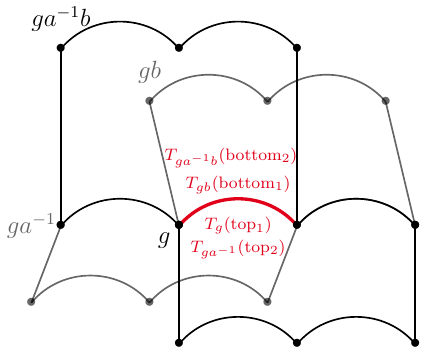}
  \caption{Illustration of the neighbor rules for $BS(2,2)$.}
  \label{fig:tiling}
\end{figure}
See \cref{fig:tiling} for an illustration of these rules.\\
The set of all valid tilings for a tileset $\tau$ forms an SFT $X^{\tau}$, since the tileset gives a finite number of local constraints based on a finite alphabet. In general, it is not necessarily simpler to consider Wang tilesets instead of local constraints on the Cayley graph; however, in \cite{BS} the construction heavily uses the visual representation of tiles with numbers on the top and bottom that encode a multiplication by a real number, and this article will do the same.

\subsection{Substitutions}
\label{subsec:substitutions}

Let $\A^*$ be the set of (finite) words over $\A$.
A \emph{substitution} (or morphism) is a map $s\colon \A \rightarrow \A^*$. We say it is \emph{uniform} of size $n \in \N$ if for every $a \in \A, |s(a)|=n$.
The substitution $s$ can be extended to $\A^*$ by applying it to each letter of the word and concatenating the resulting words.
We can also extend $s$ to $\A^{\N_0}$ (resp. $\A^\N$) by doing the same, concatenating infinitely many words, the first letter of the first word being at position $0$ (resp. $1$).
Finally, $s$ can be extended to \emph{pointed} biinfinite words (so on $\A^\Z$ in a certain way) as illustrated in \cref{fig:biinf}.

\begin{figure}[ht]
  \centering
  \includegraphics[scale=0.7]{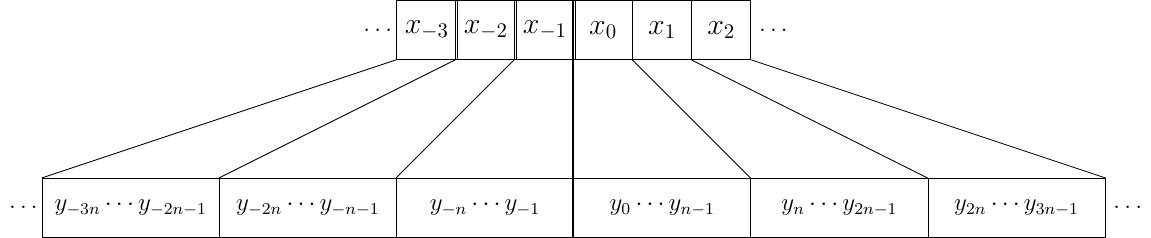}
  \caption{How a substitution is applied to a biinfinite word: $y=s(x)$ with $s$ a uniform substitution of size $n$}
  \label{fig:biinf}
\end{figure}

For $a\in \A$ with $s(a)=aw$ with $w\in\A^*$,
we define the \emph{positive infinite iteration} of $s$ on $a$:
\[ \overrightarrow{{s}^\omega}(a) = a ~ w ~ s(w) ~ s^2(w) \cdots \in \A^{\N_0} .\]
In the same way, we define the \emph{negative infinite iteration} of $s$ on $a$:
\[ \overleftarrow{~^\omega\hspace{-0.25em} s}(a) = \cdots s^2(w) ~ s(w) ~ w ~ a  \in \A^{\Z^-} .\]

For a word $u$ (possibly biinfinite), we define its factor complexity \[ P_u(n) = |\{w \in \A^n \mid w \sqsubset u\}| \] where $w \sqsubset u$ indicates that $w$ is a subword of $u$. The factor complexity of a biinfinite word is bounded if and only if that word is periodic \cite{cassaigne_nicolas_2010}, that is, if it is made of the same finite word concatenated infinitely many times.

A \emph{fixpoint} of a substitution $s$ is a (possibly biinfinite) word $u$ such that $s(u) = u$.

%%%%%%%%%%%%%%%%%%%%%%%%%%%%%%%%%%%%%%%%%%%%%%

\section{On a construction by Aubrun and Kari}
\label{sec:AK}

In \cite{BS}, Aubrun and Kari provide a weakly aperiodic Wang tileset on Baumslag-Solitar groups, with a proof focusing on the specific case of $BS(2,3)$, for which they also present a period for one specific configuration, implying that the corresponding Wang tileset is not strongly aperiodic.

A more general version of the construction can be found in \cite{BSdetailed}. We repeat most of it here for the sake of completeness, since we study that construction in more details to obtain additional results: we will show that it yields a weakly but not strongly aperiodic SFT on any $BS(m,n)$ with $|m| \neq 1$ and $|n| \neq 1$, and a strongly aperiodic SFT on any $BS(1,n)$.

\subsection{Aubrun and Kari's construction}
\label{subsec:theconstruction}

Aubrun and Kari's construction works by encoding orbits of piecewise affine maps applied to real numbers.
We will only apply their construction for piecewise linear maps, and begin this section with the necessary definitions.

\subsubsection{Definitions}

\begin{definition}[(Representation by a sequence)]
  \label{def:balanced_repr}
  Let $i\in\Z$. 
  We say that a binary biinfinite sequence $(x_k)_{k \in \Z} \in \{i, i+1\}^{\Z}$ \emph{represents} a real number $x \in [i,i+1]$ if there exists a growing sequence of intervals $I_1 \subset I_2 \subset \dots \subseteq \Z$ of size at least $1, 2, \dots$ such that:
  \[
  \lim\limits_{k \to +\infty} \dfrac{\sum_{j \in I_k} x_j}{|I_k|} = x.
  \]
\end{definition}
Note that if $(x_k)_{k \in \Z}$ is a representation of $x$, all the shifted sequences $(x_{l+k})_{k \in \Z}$ for every $l \in \Z$ are also representations of $x$. Note that a sequence $(x_k)_{k \in \Z}$ can \textit{a priori} represent several distinct real numbers since different choices of interval sequences may make it converge to different points. A sequence always represents at least one real number by compactness of $[i,i+1]$.

Then, we define a generalization of piecewise linear maps: multiplicative systems. 
The main difference with piecewise linear maps is that points may have several images as definition intervals of different pieces might overlap.

\begin{definition}[Multiplicative system]
  A \emph{multiplicative system} is a finite set of non-zero linear maps \[ \mS = \{ f_1: I_1\rightarrow I'_1, \hdots, f_k: I_k\rightarrow I'_k \} \]
  with $I_i$ and $I'_i$ closed intervals of $\R$.
  Its \emph{inverse} is defined to be 
  \[ \mS^{-1} = \{ f_1^{-1}: I'_1\rightarrow I_1, \hdots, f_k^{-1}: I'_k\rightarrow I_k \} .\]
  The \emph{image} of $x\in\bigcup_i I_i$ is the set
  \[ \mS(x) = \{  y\in \bigcup_j I'_j \mid \exists i, f_i(x) = y \} .\]
  The \emph{$k$-th iteration} of $\mS$ on $x\in\bigcup_i I_i$ is then:
  \[
  \begin{cases}
    \{ y\in\R \mid \exists i_1, \hdots, i_k, f_{i_k}\circ \cdots \circ f_{i_1}(x) = y \}    & \text{if } k>0 \\
    x     & \text{if } k=0 \\
    \{ y\in\R \mid \exists i_{-1}, \hdots, i_{-k}, f^{-1}_{i_{-k}}\circ \cdots \circ f^{-1}_{i_{-1}}(x) = y \}  & \text{if } k<0 \\
  \end{cases}
  .
  \]
\end{definition}

Note that if none of the intervals overlap, $\mS$ can be represented as a piecewise linear function and the definition of inverse and iteration coincide with the usual ones.

\begin{definition}[(Immortal and periodic points)]
  Let $\mS = \{ f_1: I_1\rightarrow I'_1, \hdots, f_k: I_1\rightarrow I'_k \}$ be a multiplicative system.
  The real $x\in \R$ is \emph{immortal} if for all $k\in\Z$,
  \[ \mS^k(x) \cap \bigcup_i I_i \neq \emptyset \]
  A \emph{periodic point} for this system is a point $x\in\R$ such that there exists $k\in\N^*$ such that
  \[ x\in \mS^k(x) . \]
\end{definition}

\begin{definition}[(Level)]
  The \emph{level} of $g\in BS(m,n)$ is the set $\mathcal{L}_g = \{ g a^k \mid k \in \Z \}$.
\end{definition}

When considering a tiling of $BS(m,n)$,
given a line of tiles located between levels $\mathcal{L}_g$ and $\mathcal{L}_{gb^{-1}}$, we talk about the upper side of the line to refer to level $\mathcal{L}_g$, and the lower side of the line to refer to level $\mathcal{L}_{gb^{-1}}$.

\begin{definition}[(Height)]
  The \emph{height} of $g\in BS(m,n)$ is, for any way of writing it as a word in $\{a,b,a^{-1},b^{-1}\}^*$, its number of $b$'s minus its number of $b^{-1}$'s; it is denoted as $||g||_b$. 
\end{definition}
Since the only basic relation in $BS(m,n)$ uses one $b$ and one $b^{-1}$, all writings of $g$ as a word give the same height. Furthermore, it is actually the height of all elements in its level.

\begin{definition}[(Multiplying tileset)]
  A set of tiles $\tau$ \emph{multiplies} by $q\in\Q$ if for any tile $(t_1,\dots,t_m,l,r,b_1,\dots,b_n)\in\tau$ (see \cref{fig:tiles_BSmn} for the notation),
  \begin{equation}\label{eq:def_multiply}
  q\frac{t_1 + \cdots + t_m}{m} + l = \frac{b_1 + \cdots + b_n}{n} + r .
  \end{equation}
\end{definition}

Let $\tau$ be a tileset multiplying by $q\in\Q$.
If we consider a line of $N$ tiles of $\tau$ next to each other without tiling errors (as defined in \cref{subsec:tilings}), as left and right colors match, we can average \cref{eq:def_multiply}:
\begin{equation}\label{eq:k_multiply}
  q t + \frac{l}{N} = b + \frac{r}{N} .
\end{equation}
where $t$ is the average of the top labels of the line and $b$ the average of the bottom ones.
Therefore, if an infinite line has its upper side representing $x\in \R$ and its lower side representing $y\in\R$, taking the limit of \cref{eq:k_multiply} on a well chosen sequence of intervals gives:
\[ qx = y. \]
Hence the name of \emph{multiplying tileset} for $\tau$.

\subsubsection{A multiplying tileset}

Let us define, in a fashion similar to \cite{BS}, a couple of useful functions to build a multiplying tileset.
Let $\alpha_{m,n}:\{a,b,a^{-1}, b^{-1}\}^*\rightarrow \R$ (or just $\alpha$ when $m$ and $n$ are clear) be defined by the recursion:
\begin{align*}
  \begin{cases}
    \alpha(\epsilon)=0 \text{ where } \epsilon \text{ is the empty word}\\
    \alpha(wb)=\alpha(wb^{-1})=\alpha(w)\\
    \alpha(wa)=\alpha(w) + \left(\frac{n}{m}\right)^{||w||_b}\\
    \alpha(wa^{-1})=\alpha(w) - \left(\frac{n}{m}\right)^{||w||_b}.
  \end{cases}
\end{align*}
The map $\alpha$ can be extended to elements of $BS(m,n)$, due to the fact that $\alpha(uba^mb^{-1}v) = \alpha(ua^nv)$ for any pair of words $u$ and $v$ in $\{a,b,a^{-1}, b^{-1}\}^*$: $\alpha(g)$ is then $\alpha(w)$ for any word representing $g$ in the group.

% $\alpha$ is also well defined on $BS(m,n)$ instead of $\{a,b,a^{-1}, b^{-1}\}^*$ because it respects the relation of $BS(m,n)$ that generates all the others: $\alpha(wba^mb^{-1}) = \alpha(wa^n)$ for any word $w$. In some sense, $\alpha(g)$ yields a writing of $g$ in a ``$\frac{n}{m}$ base'', where $a$ and $a^{-1}$ correspond to incrementing or decrementing the current digit, and $b$ and $b^{-1}$ correspond to focusing on the digit to its left or right.

\noindent
Now, we define $\Phi:BS(m,n)\rightarrow \R^2$ as follows:
\[ \Phi(g) = \left( \alpha(g), ||g||_b \right) .\]
The function $\Phi$ can be seen as a projection of every element of $BS(m,n)$ on the euclidean plane $\R^2$.

Finally, let $\lambda:BS(m,n)\rightarrow\R$ be defined as
\[ \lambda(g) = \frac{1}{m} \left( \frac{m}{n} \right)^{||g||_b} \alpha(g) .\]

Let $q\in\Q$ and $I$ an interval, let
\begin{align*}
 t_j(g,x) &:= \lfloor \left(m\lambda(g)+j \right)x \rfloor - \lfloor \left(m\lambda(g)+(j-1) \right)x \rfloor \text{~~for }j=1\dots m\\   
 b_j(g,x) &:= \lfloor \left(n\lambda(g)+j \right)qx \rfloor - \lfloor \left(n\lambda(g)+(j-1) \right)qx \rfloor \text{~~for }j=1\dots n\\               
 l(g,x) &:= \frac{1}{m}q\lfloor m\lambda(g)x \rfloor - \frac{1}{n}\lfloor n\lambda(g)qx \rfloor \\
 r(g,x) &:=  \frac{1}{m} q \lfloor \left(m\lambda(g)+m\right)x \rfloor - \frac{1}{n}\lfloor \left(n\lambda(g)+n\right)qx \rfloor
\end{align*}
Then, we define the tileset $\tau_{q,I}$ as:
\[ \tau_{q,I} = \{ (t_1(g,x), \hdots, t_m(g,x), l(g,x), r(g,x), b_1(g,x), \hdots, b_m(g,x)) \mid g\in BS(m,n), x\in I \} . \]

\begin{figure}[ht]
  \centering
  \includegraphics[scale=0.7]{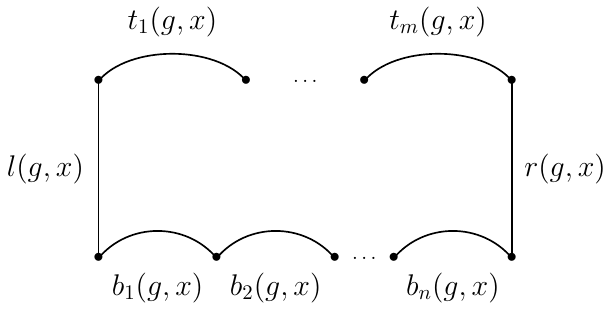}
  \caption{One tile from the tileset $\tau_{q,I}$.}
  \label{fig:tileset_nathalie}
\end{figure}

One can show that \cref{eq:def_multiply} holds for these tiles.

\begin{prop}[{\cite[Proposition 6]{BSdetailed}}]
  \label{prop:multiply}
  Let $a\in\Z$, for any $I\subseteq [a,a+1]$, $\tau_{q,I}$ is a tileset that multiplies by $q$.
\end{prop}

Let us define the \emph{balanced representation of $x$}, which is the biinfinite sequence defined for any $z\in\R$ by
\[ B_j(x,z) = \left\lfloor (z+j)x \right\rfloor - \left\lfloor (z+j-1)x \right\rfloor  . \]
Note that $B_j(x,z)$ does not depends on $z$, and can only take two values: $\lfloor x \rfloor$ or $\lfloor x \rfloor+1$.

\begin{prop}
  \label{prop:balanced_tileset}
  Let $q\in\Q, a\in\Z$ and $I\subseteq [a, a+1]$. The upper side of any tile in $\tau_{q, I}$ is of the form
  \[ B_j(x,m\lambda(g)), B_{j+1}(x,m\lambda(g)), \hdots, B_{j+m}(x,m\lambda(g))  ;\]
  for $x\in I, j\in\Z, g\in BS(m,n)$. In particular, its labels are in $\{a, a+1\}$.
  The lower side is of the form
  \[ B_j(qx,m\lambda(gb^{-1})), B_{j+1}(qx,m\lambda(gb^{-1})), \hdots, B_{j+n}(qx,m\lambda(gb^{-1})). \]
\end{prop}

\begin{proof}
  Rewriting the top labels using balanced representation yields
  \[B_j(x,m\lambda(g)), B_{j+1}(x,m\lambda(g)), \hdots, B_{j+m}(x,m\lambda(g)).\]
  Since each $B_j(x,m\lambda(g))$ is either $\floor{x}$ or $\floor{x}+1$, and $x \in [a,a+1]$, one obtains labels in $\{a,a+1\}$.
  For the bottom side, note that $\lambda(gb^{-1}) = \frac{n}{m}\lambda(g)$, which gives labels
  \[B_j(qx,m\lambda(gb^{-1})), B_{j+1}(qx,m\lambda(gb^{-1})), \hdots, B_{j+n}(qx,m\lambda(gb^{-1})).\]
\end{proof}

% Note: proposition inutile, on s'en sert juste pour prouver que c'est fini. Mais bon, les balanced sequences c'est cool.

For our purpose we need to have a finite tileset, because subshifts use a finite alphabet.

\begin{prop}
  Let $a\in\Z$, for any $I\subseteq [a,a+1]$ the tileset $\tau_{q,I}$ is finite.
\end{prop}

\begin{proof}
  \cref{prop:balanced_tileset} gives us that there are finitely many top and bottom labels. It remains to prove that there are also finitely many left and right labels.

  First of all, one can check that $\lambda(ga^m) = \lambda(g)+1$, and so $l(ga^m, x) = r(g,x)$. Consequently, we simply have to prove that $l$ lies in a finite set.
  Let $q = \frac{q_1}{q_2}$ with $q_1,q_2\in\Z$, and write
  \[  l(g,x) = \frac{n q\lfloor m\lambda(g)x \rfloor - m \lfloor n\lambda(g)qx \rfloor}{mn} = \frac{n q_1 \lfloor m\lambda(g)x \rfloor - m q_2 \lfloor n\lambda(g)qx \rfloor}{mnq_2} .\]
  Since its numerator is an integer bounded by $-nq_1=:k_1$ from below and $mq_2=:k_2$ from above using usual inequalities on the floor function, we have that for any $g\in BS(m,n), x\in I$, $l(g,x)$ is in the finite set
  \[ \left\{  \frac{k_1}{mnq_2}, \frac{k_1+1}{mnq_2}, \hdots, \frac{k_2}{mnq_2} \right\} .\]
\end{proof}

Thanks to the multiplying property of $\tau_{q,[a,a+1]}$, we can use it to encode multiplicative systems in such a way that non-empty tilings corresponds to immortal points of the system. 
If we modify the tileset a little bit, we can encode multiplying systems defined on other intervals than $[a,a+1].$

\begin{theorem}
\label{th:superthm}
  Let $\mS = \{ f_1, \hdots, f_N \}$ be a multiplicative system with 
  \begin{align*}
    f_i: &I_i \rightarrow \R\\
         &x \mapsto q_i x    ,
  \end{align*}
  $q_i\in\Q$ and $I_i$ interval with rational bounds included in some $[a_i,a_i+1], a_i\in\Z$.
  We can explicitly and algorithmically build an SFT $Y_\mS$ with the following properties:
  \begin{enumerate}
    \item any top of a line of tiles in a configuration $y \in Y_\mS$ represents at least one real $x \in \bigcup_i I_i$. \label{itemthrep}
    \item if the top of a line of tiles represents a real $x \in \bigcup_i I_i$, then the bottom of that line represents a real in $\mS(x)$; \label{itemthS}
    \item $Y_\mS \neq \emptyset$ if and only if $\mS$ has an immortal point; \label{itemthnonempt}
  \end{enumerate}
\end{theorem}

\begin{proof}
  We build a tileset $\tau$ performing the computation by the linear functions $f_i: [a_i,a_i+1] \rightarrow \R$; i.e. the linear maps with bigger intervals than the ones defining $\mS$.
  In order to encode the multiplication correctly, one cannot simply take the union of all $\tau_i := \tau_{q_i, [a_i,a_i+1]}$, because tiles coming from different $f_i$ could be mixed on a single line.
  In order to "synchronize" the computations on every line, we create a product alphabet with the left and right colors of the tiles, and the number of the current function being used. This ensures that one line can have tiles from only one of the $\tau_{q_i, [a_i,a_i+1]}$.
  Formally, 
  \[ \tau = \left\{ (t_1,\dots,t_m,(l, i),(r,i),b_1,\dots,b_n) \mid (t_1,\dots,t_m,l,r,b_1,\dots,b_n)\in \tau_{q_i,[a_i,a_i+1]} \right\} .\]
  This way, we can interpret any line of a tiling by $\tau$ as being "of color" $i$ for some $i \in \{1,\dots,N\}$.

  Next, we restrict the intervals of reals that can be represented in each line of the SFT.
  Let us write $I_i = [a_i + \frac{d_1^i}{e_1^i}, a_i + 1 - \frac{d_2^i}{e_2^i}]$, for each $i$.
  $Y_\mS$ will be $X^\tau$ with the following additional local constraints:
  on each line of color $i$, we force the upper side labels to respect
  \begin{align}
    \label{eq:Ys_constraint_1}
  &\text{\textbullet every $e_1^i$ consecutive labels must contain at least $d_1^i$ labels $a_i+1$;} \\
    \label{eq:Ys_constraint_2}
  &\text{\textbullet every $e_2^i$ consecutive labels must contain at least $d_2^i$ labels $a_i$.}
  \end{align}
  Recall that the top labels of any line of color $i$ only has labels $a_i$ and $a_i+1$ by definition of $\tau_i$.
  From this, we can also deduce
  \begin{align*}
  &\text{\textbullet every $e_1^i$ consecutive labels must contain at most $e_1^i - d_1^i$ labels $a_i$;} \\
  &\text{\textbullet every $e_2^i$ consecutive labels must contain at most $e_2^i - d_2^i$ labels $a_i+1$.}
  \end{align*}
  Since these constraints are on $e_1^i$ or $e_2^i$ consecutive labels, and any other constraint from the $\tau_i$'s is on neighboring tiles, $Y_\mS$ is an SFT.

  \medskip

  First, assume that $Y_\mS$ is non-empty and contains some configuration $y$.
  Then, the sequence at the top of each level $\mathcal{L}_{b^k}, k\in \Z$ represents at least one real $x_k$, since it is a sequence made of at most two integers.
  Thanks to the multiplying property of each $\tau_i$ (\cref{prop:multiply}), and the fact that the local rules of every line impose that any real represented belongs to an interval $I_i$, $(x_k)_{k\in\Z}$ is an infinite orbit of $\mS$.
  Indeed, consider a real $x$ represented by the upper side of a line of color $i$. We prove that $x\in I_i = [a + \frac{d_1}{e_1}, a + 1 - \frac{d_2}{e_2}]$ (we drop the $i$ in this paragraph to make notations lighter).
  Let us write $(J_l)_{l\in\N}$ the intervals from \cref{def:balanced_repr} (denoted as $(I_i)$ there) on which we compute the mean for the represented real, and let $r_l$ denote the proportion of $a$'s over $(a+1)$'s in the line representing $x$ restricted to $J_l$.
  Thanks to the two conditions \cref{eq:Ys_constraint_1} and \cref{eq:Ys_constraint_2} (and the two we deduced), we have, for all $l$ in $\N$,
  \begin{equation}
    \label{eq:r}
    \frac{d_2}{e_2 - d_2} \leq r_l \leq \frac{e_1-d_1}{d_1} .
  \end{equation}
  Moreover, since $x$ is the limit of the means computed on each $J_l$, one can show that 
  \begin{equation*}
    x = \lim_{l\rightarrow\infty} a + \frac{1}{1+r_l} .
  \end{equation*}
  Using the left inequality of \cref{eq:r} gives that 
  \[ \frac{1}{1+r_l} \leq \frac{e_2-d_2}{e_2} = 1 - \frac{d_2}{e_2} .\]
  So
  \[ x \leq a + 1 - \frac{d_2}{e_2}  .\]
  Similarly, the right inequality of \cref{eq:r} gives
  \[ x\geq a+\frac{d_1}{e_1}, \]
  and so $x\in I_i$.
  This notably ensures \cref{itemthrep} and, by the use of the $\tau_i$'s, \cref{itemthS}.

  \medskip

  Now, assume that $S$ has an immortal point $x$. 
  Then there exists a sequence $(i_k)_{k\in\Z} \in \{1, \dots, N\}^\Z$ such that if we define
  \[ x_k = 
  \begin{cases}
    x                & \text{if } k=0 \\
    f_{i_{k-1}}(x_{k-1}) = q_{i_{k-1}} x_{k-1}  & \text{if } k>0 \\
    f^{-1}_{i_{k}}(x_{k+1}) = \frac{1}{q_{i_{k}}} x_{k+1} & \text{if } k<0 \\
  \end{cases} , \]
  then for all $k\in\Z, x_k\in \bigcup_i I_i$.
  For every $g\in BS(m,n)$, we place at $g$ a tile from the tileset $\tau_{i_k}$ with $k = -||g||_b$, with colors:
  \begin{align*}
   t_j: t_j(g,x_k) & = \lfloor \left(m\lambda(g)+j \right)x_k \rfloor - \lfloor \left(m\lambda(g)+(j-1) \right)x_k \rfloor \text{~~for } j=1\dots m\\   
   b_j: b_j(g,x_k) & = \lfloor \left(n\lambda(g)+j \right)q_{i_k} x_k \rfloor - \lfloor \left(n\lambda(g)+(j-1) \right)q_{i_k} x_k \rfloor \text{~~for } j=1\dots n\\               
   l: l(g,x_k) & = \left(\frac{1}{m}q_{i_k}\lfloor m\lambda(g)x_k \rfloor - \frac{1}{n}\lfloor n\lambda(g)q_{i_k}x_k \rfloor ,~i_k \right)\\
   r: r(g,x_k) & =  \left(\frac{1}{m} q_{i_k} \lfloor \left(m\lambda(g)+m\right)x_k \rfloor - \frac{1}{n}\lfloor \left(n\lambda(g)+n\right)q_{i_k}x_k \rfloor ,~i_k\right)
  \end{align*}

  These tiles are obviously from the tileset $\tau$. Recall that we have
  \begin{equation}
    \label{eq:lambdaa} 
    \lambda(ga^m) = \lambda(g) + 1 , 
  \end{equation}

  \[ \lambda(gb) = \frac{m}{n}\lambda(g) . \]
  And therefore,
  \begin{align*}
  l(ga^m,x_k) 
  & = \left(\frac{1}{m}q\lfloor m\lambda(ga^m)x_k \rfloor - \frac{1}{n}\lfloor n\lambda(ga^m)qx_k \rfloor ,~i_k \right) \\
  & = \left(\frac{1}{m}q_{i_k}\lfloor m(\lambda(g)+1)x_k \rfloor - \frac{1}{n}\lfloor n(\lambda(g)+1)q_{i_k}x_k \rfloor ,~i_k \right) = r(g, x_k) ,
  \end{align*}
  and for any $j \in \{1,\dots,m\}, p \in \{1,\dots,n\}$,
  \begin{align*}
  & b_p(ga^{j-p}b,x_{-||ga^{j-p}b||_b}) \\
  & = \left\lfloor \left(n\lambda(ga^{j-p}b)+p \right)q_{i_{-||ga^{j-p}b||_b}}x_{-||ga^{j-p}b||_b} \right\rfloor - \left\lfloor \left(n\lambda(ga^{j-p}b)+(p-1) \right)q_{i_{-||ga^{j-p}b||_b}}x_{-||ga^{j-p}b||_b} \right\rfloor \\
  & = \left\lfloor \left(n\frac{m}{n}\lambda(ga^{j-p})+p \right)q_{i_{-||g||_b-1}} x_{-||g||_b-1} \right\rfloor - \left\lfloor \left(n\frac{m}{n}\lambda(ga^{j-p})+(p-1) \right) q_{i_{-||g||_b-1}} x_{-||g||_b-1} \right\rfloor\\
  & = \left\lfloor \left(m(\lambda(g)+\frac{j-p}{m})+p \right)q_{i_{-||g||_b-1}} x_{-||g||_b-1} \right\rfloor - \left\lfloor \left(m(\lambda(g)+\frac{j-p}{m})+(p-1) \right) q_{i_{-||g||_b-1}} x_{-||g||_b-1} \right\rfloor\\
  & = \left\lfloor \left(m\lambda(g)+j \right)q_{i_{-||g||_b-1}} x_{-||g||_b-1} \right\rfloor - \left\lfloor \left(m\lambda(g)+(j-1) \right) q_{i_{-||g||_b-1}} x_{-||g||_b-1} \right\rfloor\\
  & = \left\lfloor \left(m\lambda(g)+j \right)x_{-||g||_b} \right\rfloor - \left\lfloor \left(m\lambda(g)+(j-1) \right)x_{-||g||_b} \right\rfloor\\
  & = t_j(g, x_{-||g||_b})
  \end{align*}

  It remains to show that the labels at the top of every line follow the conditions (\ref{eq:Ys_constraint_1}) and (\ref{eq:Ys_constraint_2}).
  Fix some $g \in BS(m,n)$ and consider the level $\mathcal{L}_g$. Fix $k=-||g||_b$, denote $x:=x_k$, $i:=i_k$ and drop the $i$ in the other variables to simplify notations.
  For $j\in\Z$, we write $w_j$ for the label at position $j$ of the considered level, with $t_1(g,x_k)$ being position $1$.
  Remark that
  \[ w_j = \lfloor \left(m\lambda(g)+j \right)x \rfloor - \lfloor \left(m\lambda(g)+(j-1) \right)x \rfloor \]
  holds for all $j\in\Z$ thanks to \cref{eq:lambdaa}.
  Assume that $x\geq a+\frac{d_1}{e_1}$.
  Let us write $N_{j_0}$ the number of labels $a$ that appear in the word $u = w_{j_0+1}w_{j_0+2} \cdots w_{j_0+e_1}$.
  If we sum all the labels of $u$, we have on the one hand
  \[ \sum_{j=j_0+1}^{j_0+e_1} w_j = N_{j_0} a + (e_1 - N_{j_0}) (a+1) = e_1(a+1) - N_{j_0} .\]
  And on the other hand,
  \[ \sum_{j=j_0+1}^{j_0+e_1} w_j = \lfloor \left(m\lambda(g)+ j_0+e_1 \right)x \rfloor - \lfloor \left(m\lambda(g)+ j_0 \right)x \rfloor .\]
  Therefore,
  \[ e_1(a+1) - N_{j_0} > \left(m\lambda(g)+ j_0+e_1 \right)x -1 - \left(m\lambda(g)+ j_0 \right)x = e_1 x - 1 \geq e_1 a+d_1-1 ,\]
  which can be rearranged to
  \[ N_{j_0} \leq e_1 - d_1 ,\]
  which implies (\ref{eq:Ys_constraint_1}).
  In the same way, if we assume $x\leq a+1-\frac{d_2}{e_2}$, we can show $N_{j_0} \geq d_2$, which is exactly (\ref{eq:Ys_constraint_2}).
  This shows \cref{itemthnonempt}.
\end{proof}

\begin{remark}
  Instead of considering a multiplicative system with rational bounds for the intervals, one can dilate them and encode a dilated system with integer bounds, which is equivalent to the original. For instance, instead of using a linear function $f$ on $[a+\frac{d_1}{e_1},a+1-\frac{d_2}{e_2}]$, one can consider that function $f$ multiplied by $e_1e_2$, on $[ae_1e_2 + d_1e_2, ae_1e_2 + e_1e_2 - d_2e_1]$ which has integer bounds.
  However, although this would yield a shorter proof, the results would only hold through some sort of conjugation, and we wanted to effectively build a Wang tileset for any multiplicative system.
\end{remark}

Note that we only build a multiplying tileset in the current paper. 
Aubrun and Kari provided the details of the encoding of any finite set of affine maps into a tileset of $BS(m,n)$ in \cite{BSdetailed}.

\subsection{A weakly aperiodic SFT on \texorpdfstring{$BS(m,n)$}{BS(m,n)}}
\label{subsec:weakmn}

The SFT $Y_\mS$ previously defined on a given $BS(m,n)$ is also linked to the periodicity of $\mS$.

\begin{theorem}
  If $\mS$ has no periodic point, then $Y_\mS$ is a weakly aperiodic SFT.
\end{theorem}

\begin{proof}
We prove the contrapositive.

Assume that $Y_S$ has a strongly periodic configuration $y$, i.e. $|Orb_{BS(m,n)}(y)|=i$ with $i \in \N$.
In particular, each set $\{ ga^{k} g^{-1} \cdot y \mid k\in\Z \}$ is finite of cardinality lesser than $i$, and therefore for every $g\in BS(m,n)$, there exists $k_g \leq i$ such that $ga^{k_g} g^{-1} \cdot y = y$.
If we define $p=i!$, we obtain that for all $g \in BS(m,n)$, $ga^{p} g^{-1} \cdot y = y$.

Let $g \in BS(m,n)$. Let $h \in \mathcal{L}_g$, i.e. there exists some $k \in \Z$ so that $h = ga^k$. Then
\begin{align*}
y_h & = (ga^{p} g^{-1} \cdot y)_h\\
& = y_{ga^{-p} g^{-1}h}\\
& = y_{ga^{k-p}}\\
& = y_{ha^{-p}}
\end{align*}
which means that the level $\mathcal{L}_g$ is $p$-periodic.

Therefore any level is $p$-periodic, and consequently represents a unique rational number $\frac{c}{p}$. Since the alphabet of $Y_\mS$ is finite, there are only finitely many different such rationals. Consequently, there exist two levels $\mathcal{L}_{gb^l}$ and $\mathcal{L}$ with $l>0$ that represent the same rational number $x$.
By the multiplicative property of $Y_\mS$,
\[ x \in\mS^l(x), \]
which means that $x$ is a periodic point for $\mS$.
\end{proof}

The consequence of this theorem is that, to obtain a weakly aperiodic SFT on $BS(m,n)$, we only need to explicitly build a multiplicative system with an immortal point but no periodic points.
For the rest of this section, let
\begin{align*}
  f_1\colon[\frac{1}{3},1] &\rightarrow[\frac{2}{3},2] \\
  x & \mapsto 2x
\end{align*}

\begin{align*}
  f_2\colon[1,2] &\rightarrow[\frac{1}{3},\frac{2}{3}]  \\
  x & \mapsto \frac{1}{3}x
\end{align*}
and
\[ \mS_0 = \{ f_1, f_2 \} .\]

It is easy to see that this system has immortal points (in fact, all points of $[\frac{1}{3}, 2]$ are immortal).

\begin{prop}
  $\mS_0$ has immortal points.
\end{prop}

Additionally, since $2$ and $3$ are relatively prime:

\begin{prop}
  $\mS_0$ has no periodic points.
\end{prop}

\begin{corollary}
  $Y_{\mS_0}$ is non-empty and weakly aperiodic.
\end{corollary}

However, this construction does not avoid weakly periodic configurations when $m$ and $n$ are not 1, as already remarked by Aubrun and Kari.

\begin{prop}
  \label{prop:AKperiod}
  For any $m,n > 1$, $Y_{\mS_0}$ on $BS(m,n)$ contains a weakly periodic tiling, with period $p = bab^{-1}aba^{-1}b^{-1}a^{-1}$.
\end{prop}

To prove it, we need the following lemma:

\begin{lemma}
  \label{lemmap}
  Let $p = bab^{-1}aba^{-1}b^{-1}a^{-1}$.
  For any $g \in BS(m,n)$, $\alpha(pg) = \alpha(g)$.
\end{lemma}

\begin{proof}
  Since $||p||_b = 0$, using the definition of $\alpha$ it is easy to show that $\alpha(pg) = \alpha(p) + \alpha(g)$ by recurrence on the length of $g$.
  Then, $\alpha(p) = \frac{n}{m} + 1 - \frac{n}{m} - 1 = 0$.
\end{proof}

\begin{proof}[Proof of \cref{prop:AKperiod}]
We happen to have built a periodic configuration already: the one from the proof of \cref{th:superthm}.

Indeed, let $x$ be any real in $[\frac{1}{3}, 2]$, since they are all immortal for $\mS_0$. Then there exists a sequence $(i_k)_{k\in\Z} \in \{1, 2\}^\Z$ such that if we define
\[ x_k = 
\begin{cases}
  x                & \text{if } k=0 \\
  f_{i_{k-1}}(x_{k-1}) = q_{i_{k-1}} x_{k-1}  & \text{if } k>0 \\
  f^{-1}_{i_{k}}(x_{k+1}) = \frac{1}{q_{i_{k}}} x_{k+1} & \text{if } k<0 \\
\end{cases} , \]
then for all $k\in\Z, x_k\in [\frac{1}{3}, 2]$.
For every $g\in BS(m,n)$, we place at $g$ a tile from the tileset $\tau_{i_k}$ with $k = -||g||_b$, with colors:
\begin{align*}
  t_j: t_j(g,x_k) & = \lfloor \left(m\lambda(g)+j \right)x_k \rfloor - \lfloor \left(m\lambda(g)+(j-1) \right)x_k \rfloor \text{~~for } j=1\dots m\\   
  b_j: b_j(g,x_k) & = \lfloor \left(n\lambda(g)+j \right)q_{i_k} x_k \rfloor - \lfloor \left(n\lambda(g)+(j-1) \right)q_{i_k} x_k \rfloor \text{~~for } j=1\dots n\\               
  l: l(g,x_k) & = \left(\frac{1}{m}q_{i_k}\lfloor m\lambda(g)x_k \rfloor - \frac{1}{n}\lfloor n\lambda(g)q_{i_k}x_k \rfloor ,~i_k \right)\\
  r: r(g,x_k) & =  \left(\frac{1}{m} q_{i_k} \lfloor \left(m\lambda(g)+m\right)x_k \rfloor - \frac{1}{n}\lfloor \left(n\lambda(g)+n\right)q_{i_k}x_k \rfloor ,~i_k\right)
\end{align*}

We already checked that the resulting tiling $y$ was in $Y_{\mS_0}$, see the proof of \cref{th:superthm}. 
It remains to show that the tiles at $g$ and at $pg$ are the same for all $g\in BS(m,n)$, to conclude that $p^{-1} \cdot y = y$, or equivalently that $p$ is a period of $y$.
This is actually surprisingly easily, considering that for any $g \in BS(m,n)$, $||pg||_b = ||g||_b$ so they use the same integer $k$ and the same real $x_k$; and $\lambda(g) = \lambda(pg)$, see \cref{lemmap}. The tiles have the same labels as a consequence of this.
  
The $y$ consequently defined is $p$-periodic. The only thing left to show is that $p$ is a nontrivial element of $BS(m,n)$ as long as $m,n > 1$. 
Since it is freely reduced and does not contain $ba^mb^{-1}$ or $b^{-1}a^nb$ as subwords, and since $BS(m,n)$ is a HNN extension $\Z*_{\alpha}$ with $\alpha\colon m\Z \rightarrow n\Z$, we can apply Britton's Lemma: $p$ cannot be the neutral element.

Therefore $Y_{\mS_0}$ contains a configuration with a nontrivial period, and consequently is not strongly aperiodic.
\end{proof}

%%%%%%%%%%%%%%%%%%%%%%%%%%%%%%%%%%%

\subsection{A deeper understanding of the configurations}
\label{subsec:deeper}

We first present additional results on this tileset $Y_{\mS_0}$ of $BS(m,n)$. Most of the ideas present in this section were already present in~\cite{DGG2014} in the context of tilings of the plane.

For a given line $\mathcal{L}_g$, we define the sequence $u_g := (y_{g a^i})_{i\in\Z}$ to be the sequence of digits on the line $\mathcal{L}_g$ (its origin depending on $g$).

Let $f$ be the following bijective continuous map:
\begin{align*}
  f\colon & ^{\textstyle{[\frac{1}{3},2]}}\Big/_{\textstyle{\frac{1}{3} \sim 2}} \rightarrow ^{\textstyle{[\frac{1}{3},2]}}\Big/_{\textstyle{\frac{1}{3} \sim 2}}\\
  f(x) = &
  \begin{cases}
    f_1(x)=2x &\quad\text{if } x \in (\frac{1}{3},1)\\
    f_2(x)=\frac{1}{3}x &\quad\text{if } x \in (1,2)\\
    \overline{2} & \quad\text{if } x = 1\\
    \frac{2}{3} & \quad\text{if } x = \overline{2}
  \end{cases}\\
\end{align*}
$f$ is strongly related to $\mS_0$ because for any $x \in [\frac{1}{3},2]$, for any $k \in \Z$,
\begin{equation}
  \label{eq:Sandf}
^{\textstyle{{\mS_0}^k(x)}}\Big/_{\textstyle{\frac{1}{3} \sim 2}} = f^k(x).
\end{equation}
due to the fact that for our particular $\mS_0, {\mS_0}^k(x) \subseteq [\frac{1}{3},2]$.

An easy consequence is the following:

\begin{lemma}
  \label{lemma-apply-f}
  Let $y \in Y_{\mS_0}$. Let $g \in BS(m,n)$.
  Let $x \in [\frac{1}{3},2]$ be a real represented by $u_g$; then $u_{gb^{-1}}$ represents $f(x)$ (which means either $\frac{1}{3}$ or $2$ if $f(x) = \overline{2}$).
\end{lemma}

\begin{proof}
  $u_g$ represents at least one such $x$ because of \cref{itemthrep} of \cref{th:superthm}. The rest is due to \cref{itemthS} of \cref{th:superthm} and \cref{eq:Sandf}.
\end{proof}

It turns out that with our choice of multiplicative system $\mS_0$, any line can represent only one real number. We need several lemmas to prove this; all inspired by \cite{DGG2014}.

\begin{lemma}
  \label{lemmasysdyn}
  Let $\phi$ be defined as follows:
  \[ \phi \colon ^{\textstyle{[\frac{1}{3},2]}}\Big/_{\textstyle{\frac{1}{3} \sim 2}} \rightarrow ^{\textstyle{[0,1]}}\Big/_{\textstyle{0 \sim 1}} \]
  \[
  \phi(x) = \dfrac{\log(x) + \log(3)}{\log(2) + \log(3)} \mod 1
  \]
  $\phi$ is a correctly defined mapping that conjugates the dynamical systems $(^{\textstyle{[\frac{1}{3},2]}}\Big/_{\textstyle{\frac{1}{3} \sim 2}}, t, f)$ and $(^{\textstyle{[0,1]}}\Big/_{\textstyle{0 \sim 1}}, t^\prime, \phi \circ f \circ \phi^{-1})$, where $t$ and $t^\prime$ are the usual topologies on the considered sets.
\end{lemma}

\begin{proof}
  Since the action considered on $^{\textstyle{[0,1]}}\Big/_{\textstyle{0 \sim 1}}$ is $\phi \circ f \circ \phi^{-1}$, one only needs to check that $\phi$ is bijective and continuous to conclude that it yields a conjugation.
  
  $\phi$ is clearly continuous everywhere except on $\overline{2}$; there, one can check that the left and right limits both lead to $\phi(\overline{2}) = 0$ which is correctly defined.
  
  $\phi^{-1}$ is defined by $\Phi(y) = \frac{6^y}{3}$ except with collapsed images for $0$ and $1$.
\end{proof}

\begin{lemma}
  \label{lemmarotation}
  The map $r := \phi \circ f \circ \phi^{-1}$ can be considered a rotation of irrational angle $\frac{\log(2)}{\log(2) + \log(3)}$ when identifying $^{\textstyle{[0,1]}}\Big/_{\textstyle{0 \sim 1}}$ and $\mathbb{S}^1$.
\end{lemma}

\begin{proof}
  For every $\alpha \in \phi((\frac{1}{3},1))$,
  \begin{align*}
    \phi \circ f \circ \phi^{-1} (\alpha) & = \phi(2 \phi^{-1} (\alpha))\\
    & = \dfrac{\log(2) + \log(\phi^{-1} (\alpha)) + \log(3)}{\log(2) + \log(3)} \text{ mod } 1\\
    & = \alpha + \dfrac{\log(2)}{\log(2) + \log(3)} \text{ mod } 1.
  \end{align*}
  Similarly, for every $\alpha \in \phi((1,2))$, one has
  \begin{align*}
    \phi \circ f \circ \phi^{-1} (\alpha) & = \phi(\dfrac{1}{3}\phi^{-1} (\alpha))\\
    & = \dfrac{\log(\phi^{-1} (\alpha))}{\log(2) + \log(3)} \text{ mod } 1\\
    & = \alpha + \dfrac{\log(2)}{\log(2) + \log(3)} \text{ mod } 1.
  \end{align*}
  Finally, $\phi \circ f \circ \phi^{-1} (0) = \phi(f(\overline{2})) = \phi(\frac{2}{3}) = \frac{\log(2)}{\log(2) + \log(3)}$.
\end{proof}

We now have the tools to prove the following key lemma:

\begin{lemma}{(Uniqueness of representation)}\label{lemmarepresent}
  For any $y \in Y_{\mS_0}$, for any $g \in BS(m,n)$, the sequence $u_g = (y_{ga^i})_{i\in\Z}$ represents a unique real number.
\end{lemma}

\begin{proof}
  Assume that $u_g$ represents two distinct reals $x$ and $z \in [\frac{1}{3},2]$. They cannot be $\frac{1}{3}$ and $2$ because $u_g$ uses only digits in $\{0,1\}$ or in $\{1,2\}$. Therefore they also are distinct reals in $^{\textstyle{[\frac{1}{3},2]}}\Big/_{\textstyle{\frac{1}{3} \sim 2}}$.
  
  For any $k \in \Z$, notice that $f^k(x) = \phi^{-1} \circ r^k \circ \phi(x)$ and same for $z$, from \cref{lemmasysdyn}.
  We will study the behavior of $\phi(x)$ and $\phi(z)$ under iterations of $r$.
  
  The angle $\frac{\log(2)}{\log(2) + \log(3)}$, of which $r$ is a rotation by \cref{lemmarotation}, is irrational. As a consequence, the sets $\{r^k \circ \phi(x) \mid k \in \N\}$ and $\{r^k \circ \phi(z) \mid k \in \N\}$ are both dense in $\mathbb{S}^1$.

  We introduce $d_{arc}(e^{2i\pi\theta}, e^{2i\pi\psi}) = m(\psi-\theta) \in [0,1)$ for $\theta,\psi \in \R$, where $m(\psi-\theta)$ is the only real in $[0,1)$ congruent to $\psi-\theta$ mod $1$.
  We call $d_{arc}$ the \emph{oriented arc distance} (measured counterclockwise) between two elements of $\mathbb{S}^1$. It is not a distance \textit{per se} since it is not symmetric and has no triangular inequality, but its basic properties will suffice here.
  Since $r$ is a rotation, it is easy to check that it preserves $d_{arc}$. Hence we have that $\forall k \in \N, d_{arc}(r^k \circ \phi(x), r^k \circ \phi(z))$ is constant equal to some $c\in[0,1[$. Up to considering $d_{arc}(r^k \circ \phi(z), r^k \circ \phi(x))$ instead, and doing the following reasoning by swapping $x$ and $z$, we can assume that $c \leq \frac{1}{2}$.

  Let us split $\mathbb{S}^1$ between $A = \phi((\frac{1}{3},1))$, $B = \phi([1,2))$, and $\{\phi(2)\}=\{\phi(\frac{1}{3})\}=\{0\}$.
  We want to show that there is some $l \in \N$ for which $r^l \circ \phi(x) \in B$ and $r^l \circ \phi(z) \in A$.

  By density of $\{r^k \circ \phi(x) \mid k \in \Z\}$, there exists some $k_0 \in \N$ such that $d_{arc}(r^{k_0} \circ \phi(x), 0) < c$ and $r^{k_0} \circ \phi(x) \in B$. We cannot have $r^{k_0} \circ \phi(z) = 0$ without contradicting the previous inequality, hence it is either in $A$ or in $B$. But if it was in $B$, then the arc from $r^{k_0} \circ \phi(x)$ to $r^{k_0} \circ \phi(z)$ would contain all of $A$. This is not possible because $|A|_{d_{arc}} > \frac{1}{2} \geq c$.

  Hence there exists $l = k_0 \in \N$ such that $r^l \circ \phi(x) \in B$ and $r^l \circ \phi(z) \in A$.

  \begin{figure}[ht]
    \centering
    \includegraphics[scale=1]{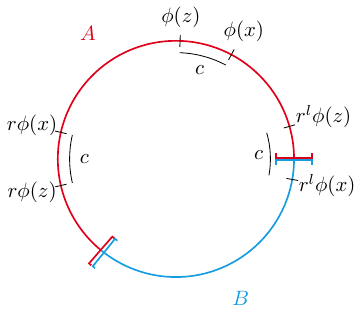}
    \caption{Preservation of the oriented arc distance $d_{arc}$ by $r$ and intersection of the arc $\left(r^l \circ \phi(x),r^l \circ \phi(z)\right)$ and the boundary between $A$ and $B$.}
    \label{fig:circle1}
  \end{figure}

  Since $r^l = \phi \circ f^l \circ \phi^{-1}$, and considering the definitions of $A$ and $B$, $f^l(z) \in (\frac{1}{3},1)$ and $f^l(x) \in [1,2)$. This would cause $f^l(z)$ to be represented by a sequence of $0$'s and $1$'s (with an infinite number of $0$'s) and $f^l(x)$ by a sequence of $1$'s and $2$'s.
  However, the SFT $Y_{\mS_0}$ is built such that a line contains only elements in $\{0,1\}$ or $\{1,2\}$, but not both (see proof of \cref{th:superthm}): this is a contradiction.

  Therefore, $x$ and $z$ must be equal, hence the uniqueness of the real number represented by a given sequence.
\end{proof}

Using previous results, we are now able to prove for $BS(m,n)$ that the real represented by the sequence $u_g$ only depends on $||g||_b$, its ``depth'' in the Cayley graph.

\begin{lemma}\label{lemmahauteur}
  Let $y \in Y_{\mS_0}$, and $e$ the identity of $BS(m,n)$.
  Let $x$ be the unique real represented by the sequence $u_e$.
  Then for every $g \in BS(m,n)$, $u_g$ represents $f^{-||g||_b}(x)$ (with a choice between $\frac{1}{3}$ and $2$, possibly different for different $g$'s, if the resulting value is $\overline{2}$).
\end{lemma}

\begin{proof}
We prove the result by reasoning on words $w \in \{a,b,a^{-1},b^{-1}\}^*$, by induction on their length. Note that we have no need of proving that different $w$'s representing the same $g$ yield the same result, since this is guaranteed by \cref{lemmarepresent}.

The result is true for $g = e$.

Suppose the result is true for words of length $n \in \N_0$. Let $w$ be a word of length $n$. Then:
\begin{itemize}
  \item $u_{wa}$ and $u_{wa^{-1}}$ represent the same real as $u_{w}$ since they are the same sequence up to an index shift;
  \item $u_{wb^{-1}}$ represents $f(f^{-||w||_b}(x))$ due to \cref{lemma-apply-f} and the induction hypothesis, which is $f^{-||wb^{-1}||_b}(x)$;
  \item suppose $u_{wb}$ represents $y$; then $u_w$ represents $f(y)$ due to \cref{lemma-apply-f}. Then we have, by induction, $y = f^{-||w||_b-1}(x) = f^{-||wb||_b}(x)$.
\end{itemize}
\end{proof}

\begin{remark}
  The previous proof heavily relies on the fact that $f$ is a bijection on $^{\textstyle{[\frac{1}{3},2]}}\Big/_{\textstyle{\frac{1}{3} \sim 2}}$, and that we do not have to differentiate between $\frac{1}{3}$ and $2$ there.
\end{remark}

\subsection{A strongly aperiodic SFT on \texorpdfstring{$BS(1,n)$}{BS(1,n)}}
\label{subsec:strong1n}

If $m$ or $n$ is equal to 1, then the previous weak period of \cref{prop:AKperiod} does not work anymore -- it is a trivial element.
In fact, we prove in this section that for $BS(1,n)$, $Y_{\mS_0}$ is strongly aperiodic.

One key property of $BS(1,n)$ is that there is a simple quasi-normal form for all its elements.

\begin{lemma}{(Quasi-normal form in $BS(1,n)$)}
\label{lemmaBS1}
For every $g \in BS(1,n)$, there are integers $k,m \in \mathbb{N}_0$ and $l \in \Z$ such that ${g = b^{-k} a^{l} b^{m}}$.
\end{lemma}
% %
\begin{proof}
From the definition of $BS(1,n)$, we have that $b a = a^n b$~(1), $b a^{-1} = a^{-n} b$~(2), $a b^{-1} = b^{-1} a^n$~(3) and $a^{-1} b^{-1} = b^{-1} a^{-n}$~(4). Consequently, taking an element of $BS(1,n)$ as a word $w$ written with $a$ and $b$, we can:
\begin{itemize}
\item Move each positive power of $b$ to the right of the word using (1) and (2) repeatedly;
\item Move each negative power of $b$ to the left of the word using (3) and (4) repeatedly;
\end{itemize}
so that we finally get a form for the word $w$ which is: $b^{-k} a^{l} b^{m}$ with $k,m \in \mathbb{N}_0$ and $l \in \Z$.
\end{proof}

\begin{remark}
\label{remarkheight}
A general normal form -- the same, with $k$ imposed to be minimal -- can be obtained from Britton's Lemma. The form obtained here is \emph{not} unique ($a = b^{-1} a^n b$ for instance), but we use it because it admits a simple self-contained proof, and it is enough for what follows: the sum $m-k$ is constant for all writings of a given group element, hence we name it ``quasi-normal''.
\end{remark}
\begin{proof}
Indeed, suppose we have $b^{-k}a^{l}b^{m} = b^{-k^\prime}a^{l^\prime}b^{m^\prime}$. Then
\begin{align*}
b^{-k}a^{l} & = b^{-k^\prime}a^{l^\prime}b^{-(m - m^\prime)}\\
& = b^{-k^\prime-(m - m^\prime)}a^{l^\prime n^{m - m^\prime}}
\end{align*}
Hence we get $a^{l^\prime n^{m - m^\prime} -l} = b^{-k + k^\prime +m - m^\prime}$.
Since it is clear that $a^i = b^j$ if and only if $i=j=0$ in $BS(1,n)$, we obtain $m- k - (m^\prime - k^\prime) = 0$ which is what we wanted.
\end{proof}

This quasi-normal form is the only thing that missed to prove the following.

\begin{theorem}
  \label{th:strong}
  For every $n \geq 2$, the Baumslag-Solitar group $BS(1,n)$ admits a strongly aperiodic SFT.
\end{theorem}
% %

\begin{proof}
  Let $y\in Y_{\mS_0}$, and $g \in Stab_{BS(1,n)}(y)$.
  Using \cref{lemmaBS1}, we can write $g = b^{-k} a^l b^{m}$ with $k, m \in \N_0, l \in \Z$.

  Let $x$ be the real represented by $u_e$.
  By \cref{lemmahauteur}, $u_g$ represents $f^{k - m}(x)$. Since $g \in Stab_{BS(1,n)}(y)$, $u_g = u_e$ and so $f^{k - m}(x) = x$ by the uniqueness of the representation from \cref{lemmarepresent}. The aperiodicity of $f$ then implies that $k = m$.

  Let us assume $l \neq 0$.
  Then $g = b^{-k} a^{l} b^{k}$ and $g^n = b^{-k} (a^n)^l b^{k}$.
  We can reduce $g^n$ to $b^{-k+1} a^l b^{k-1}$ using the relation $a^n= b a b^{-1}$.
  More generally, we notice that for any positive integer $i$, iterating the process $i$ times, we obtain that $g^{n^i} = b^{-k+i} a^l b^{k-i} \in Stab_{BS(1,n)}(y)$.

  Since for all $i$, $g^{n^i}\in Stab_{BS(1,n)}(y)$, we can obtain a contradiction with an argument similar to Prop~6. of \cite{BS}.
  We have $b^{j} a^l b^{-j} \in Stab_{BS(1,n)}(y)$ for any $j \geq -k$. This means that $u_{b^j} = u_{b^{j} a^l}$ hence $u_{b^j}$ is a $l$-periodic sequence. We have a finite number of said sequences, since they can only use digits among $\{0,1,2\}$.
  Consequently, there are $j_1\neq j_2$ such that the two levels $\mathcal{L}_{b^{j_1}}$ and $\mathcal{L}_{b^{j_2}}$ read the same sequence (up to index translation).
  These two levels represent respectively $f^{j_1}(x)$ and $f^{j_2}(x)$ due to \cref{lemmahauteur}, and since the two sequences on these levels are the same, $f^{j_1}(x) = f^{j_2}(x)$.
  This equality contradicts the fact that $f$ has no periodic point, since we had $j_1\neq j_2$.

  As a consequence, any non-trivial $g\in {BS(1,n)}$ cannot be in $Stab_{BS(1,n)}(x)$, and we finally get that $Stab_{BS(1,n)}(x) = \{e\}$: $Y_{\mS_0}$ is strongly aperiodic.
\end{proof}

Following \cref{th:strong}, a question remains: is the strong aperiodicity of Aubrun and Kari's SFT a property of the group $BS(1,n)$ itself, or does it only arise on carefully chosen SFTs, as $Y_{\mS_0}$? Is this because $BS(1,n)$ behaves like $\Z^2$ and all its weakly aperiodic SFTs are also strongly aperiodic, or does Aubrun and Kari's construction happen to be ``too much aperiodic''?
It turns out that the latter is the correct answer, as we build in the following section an SFT on $BS(1,n)$ that is weakly but not strongly aperiodic.

\section{A weakly but not strongly aperiodic SFT on \texorpdfstring{$BS(1,n)$}{BS(1,n)}}
\label{sec:weak1n}

Our weakly but not strongly aperiodic SFT will work by encoding specific substitutions into $BS(1,n)$.
Indeed, the Cayley graph of $BS(1,n)$ is very similar to orbit graphs of uniform substitutions (see for example \cite{CGS, ABM} for a definition of orbit graphs and another example of a Cayley graph similar to an orbit graph).
In this section, we find a set of substitutions that are easy to encode in $BS(1,n)$ (\cref{sec:sigma_i}), and show how to do it (\cref{sec:encoding}).

\subsection{The substitutions \texorpdfstring{$\sigma_i$}{sigma\_i}}
\label{sec:sigma_i}
Let $\A=\{0,1\}$. For $r \in \{0,\dots,n-1\}$, let $\sigma_r: \A \rightarrow \A^n$ be the following substitution:
\[ \sigma_r:
\begin{cases}
0\mapsto 0^{n-r-1}10^r \\
1\mapsto 0^n
\end{cases} .\]
We may also write $\sigma = \sigma_0$ and call the other ones the \emph{shifts} of $\sigma$.

\noindent
Note that, for $c \in \{0,1\}$ and $i \in \{0,\dots,n-1\}$, $\sigma_r(c)_i=0$ if and only if $c=0$ and $i=n-r-1$ (starting to count from $0$ the indices of the word $\sigma_r(c)$).

\noindent
All $\sigma_r(0)$ are cyclic permutations of the same finite word. Denote $\rho$ the \emph{shift action} on a biinfinite word $u$, i.e. ${\rho^j(u)}_i = u_{i+j}$, as a way to write the action of $\Z$ on $\{0,1\}^\Z$.

\begin{lemma}
  \label{lemma:exponentindex}
    For any biinfinite word $u\in\A^\Z$, any $i,r \in \{0,\dots,n-1\}$ and $j\in\Z$,
  \[(\sigma_r \circ \rho^j(u))_i = \sigma_r(u_j)_i = (\sigma_r(u))_{nj+i}.\]
\end{lemma}

\begin{proof}
  For $i \in \{0,\dots,n-1\}$, $\sigma_{r}(\rho^j(u))_i$ depends on the letter of $\rho^j(u)$ at position $0$ only, that is $u_j$ (See \cref{fig:exponentindex}), hence
    $ \sigma_{r}(\rho^j(u))_i = \sigma_{r}(\rho^j(u)_0)_i = \sigma_{r}(u_j)_i .$

    Similarly, the letter $(\sigma_r(u))_{nj+i}$ does not depend on the totality of $u$ but only on $u_j$: it is the $i$th letter of $\sigma_r(u_j)$.
\end{proof}

\begin{figure}[ht]
  \centering
  \includegraphics[scale=0.65]{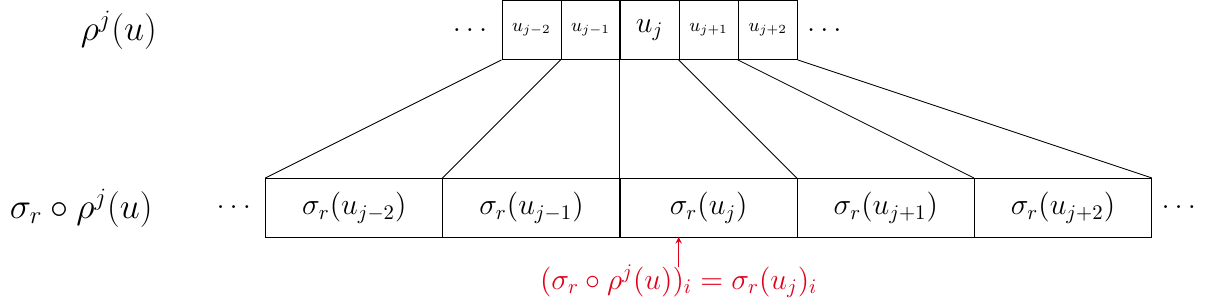}
  \caption{Illustration of \cref{lemma:exponentindex}.}
  \label{fig:exponentindex}
\end{figure}

\begin{lemma}
  \label{lemma:remainder}
  For any $r \in \{0,\dots,n-1\}$,
  \[
  \sigma_r = \rho^{r} \circ \sigma.
  \]
\end{lemma}

\begin{proof}
  Let $u \in \A^\Z$. Let $i, r \in \{0,\dots,n-1\}$ and $j \in \Z$.
  \[
  \sigma(u)_{nj+i} =
  \begin{cases}
  0 &\quad\text{if } i \neq n-1\\
  \sigma_r(u)_{nj+i-r} &\quad\text{if } i = n-1\\
  \end{cases}
  \]
  Considering that if $i \neq n-1$, $\sigma_r(u)_{nj+i-r} = 0$, we conclude that we always have $\sigma(u)_{nj+i} = \sigma_r(u)_{nj+i-r}$, and so $\sigma_r = \rho^r \circ \sigma$.
\end{proof}

\begin{lemma}
  \label{lemma:fixpoint_unique}
  For $n \geq 3$, $\sigma_1$ has a unique fixpoint.
  For $n=2$, $\sigma_1$ has no fixpoint but ${\sigma_1}^2$ has two fixpoints.
\end{lemma}

\begin{proof}
  Proposition 4 from \cite{Shallit_Wang_1999} characterizes biinfinite fixpoints of substitutions. In the present case of $\sigma_1$, \cite{Shallit_Wang_1999} states that $w = \sigma_1(w)$ if and only if $w = y.x$ with $x = \overrightarrow{{\sigma_1}^\omega}(c)$ and $y = \overleftarrow{~^\omega\hspace{-0.25em} \sigma_1}(c')$ with $\sigma_1(c) = cv$ and $\sigma_1(c') = uc'$, $u,v\in \{0,1\}^*$, $c,c' \in \{0,1\}$.
  % %
  Notice that $\sigma_1(0) = 0^{n-2}10$ and $\sigma_1(1)=0^n$, for $n \geq 3$, so the only choice for $c$ and $c'$ is $c=c'=0$.
  Then $\sigma_1$ has a fixpoint that is $\overleftarrow{~^\omega\hspace{-0.25em} \sigma_1}(0).\overrightarrow{{\sigma_1}^\omega}(0)$ and which is unique.

  For $n=2$ the same reasoning concludes that $\sigma_1$ has no fixpoint. However, since ${\sigma_1}^2(0) = 0010$ and ${\sigma_1}^2(1) = 1010$, the same reasoning also yields that ${\sigma_1}^2$ has two fixpoints that are $\overleftarrow{~^\omega\hspace{-0.25em} ({\sigma_1}^2)}(0).\overrightarrow{{({\sigma_1}^2)}^\omega}(0)$ and $\overleftarrow{~^\omega\hspace{-0.25em} ({\sigma_1}^2)}(0).\overrightarrow{{({\sigma_1}^2)}^\omega}(1)$.
\end{proof}

\begin{lemma}
  \label{lemma:fixpoint_aperiodic}
  For every $k \in \N$ and every $i_1, \dots, i_k\in \{0,\dots,n-1\}$, the fixpoints of $s = \sigma_{i_k}\circ\dots\circ\sigma_{i_1}$ are aperiodic.
\end{lemma}

\begin{proof}
  To prove the aperiodicity of a fixpoint $w$ of $s$ (in the case where such a fixpoint exists), we follow a proof from \cite{Pansiot_1986}, simplified for our specific case.

  First, let us show that the two subwords $00$ and $01$ can be found in $w$.
  \begin{itemize}
    \item For $00$,
    let us define $s' = \sigma_{i_{k-1}}\circ\dots\circ\sigma_{i_1}$.
    Then, by definition, $w = \sigma_{i_k}(s'(w))$ (by convention $s'(w)=w$ if $k=1$).
    We are going to prove that $s'(w)$ always contains a $1$. As a consequence, $w = \sigma_{i_k}(s'(w))$ contains $00$ because $\sigma_{i_k}(1) = 0^n$.
    Suppose $s'(w)=\mydots000\mydots$.
    If $k=1$, it means that $w=\mydots 000\mydots$, but then $s(w)\neq w$ so this is impossible.
    If $k=2$, then $s'=\sigma_{i_1}$ so the only way to have $s'(w)=\mydots 000\mydots$ is to have $w=\mydots 111\mydots$, but again $s(w)\neq w$.
    If $k\geq 3$, let us define $t=\sigma_{i_{k-3}}\circ\dots\circ\sigma_{i_1}$. With this notation, $w = \sigma_{i_k}\circ\sigma_{i_{k-1}}\circ\sigma_{i_{k-2}}(t(w))$. The assumption $s'(w)=\mydots 000\mydots$ causes $\sigma_{i_{k-2}}(t(w))= \mydots 111\mydots$.
    However, this is impossible since $\mydots 111\mydots$ has no antecedent by $\sigma_{i_k-2}$. Therefore $s'(w)$ must contain a $1$ and we can find $00$ in $w$.

    \item For $01$, the only way for $w$ not to contain $01$ is to be of the form $w=\mydots000\mydots$, $w=\mydots111\mydots$ or $w=\mydots 1100\mydots$.
    But it is clear that ${s(\mydots000\mydots) \neq \mydots000\mydots}$, ${s(\mydots111\mydots) \neq \mydots111\mydots}$ and ${s(\mydots1100\mydots) \neq \mydots1100\mydots}$ hence none of them can be fixpoints.
  \end{itemize}

  Hence $s(00)$ and $s(01)$ can also be found in $w$ since $s(w)=w$.
  From this, we build by induction infinitely many words with two possible right extensions. We have $s(00) \neq s(01)$; consider the largest prefix on which they agree, call it $u_2$, with $|u_2|>1$. Then both $u_20$ and $u_21$ can be found in $w$. Hence $s(u_20)$ and $s(u_21)$ can also be found in $w$.
  We have $s(u_20) \neq s(u_21)$; consider the largest prefix on which they agree, call it $u_3$, with $|u_3|>|u_2|$. Then both $u_30$ and $u_31$ can be found in $w$. Hence $s(u_30)$ and $s(u_31)$ can also be found in $w$.

  By induction, we can build subwords of $w$ as large as we want that have two choices for their last letter. Hence the factor complexity of $w$ is unbounded, and so $w$ is aperiodic (see \cref{subsec:substitutions}).
\end{proof}

\subsection{Encoding substitutions in \texorpdfstring{$BS(1,n)$}{BS(1,n)}}
\label{sec:encoding}
We now show how to encode such substitutions in SFTs of the group $BS(1,n)$ given by a tileset.
We define the tileset $\tau_\sigma$ on $BS(1,n), n \in \N, n \geq 2$, to be the set of tiles shown on \cref{fig:tileset} for all $c\in\{0,1\}$ and $i\in\{0,\dots,n-1\}$. Remark that a tile is uniquely defined by the couple $(c,i)$.
\begin{figure}[ht]
  \centering
  \includegraphics[width=5cm]{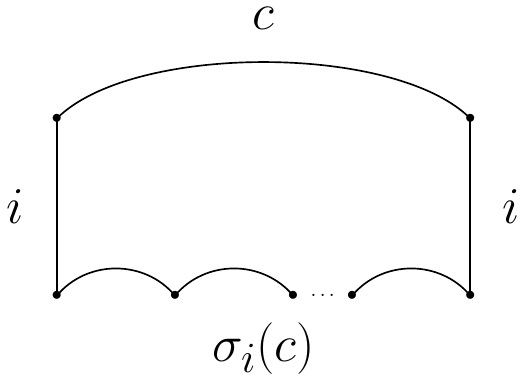}
  \caption{Tiles of $\tau_\sigma$: left and right colors are identical and equal to $i$, top color is $c$ and bottom colors are equal to $\sigma_i(c)_0,\dots,\sigma_i(c)_{n-1}$.}
  \label{fig:tileset}
\end{figure}

This tileset will be the weakly but not strongly aperiodic tileset we are looking for.
Lemmas \ref{lemma:fixpoint_unique} and \ref{lemma:fixpoint_aperiodic} study the words that can appear on levels $\mathcal{L}_g$ of the tiling, by looking at the fixpoints of $\sigma_1$.
They prove that no biinfinite word can be both a fixpoint for the $\sigma_i$'s and a periodic word, forbidding one direction of periodicity for any configuration we will encode with our tileset.
This naturally leads to the following proposition:

\begin{prop}
  \label{prop:period}
  No configuration of $X^{\tau_\sigma}$ can be $a^{k}$-periodic for any $k \in \N$.
\end{prop}

\begin{proof}
  Suppose that there is a configuration $x$ of $X^{\tau_\sigma}$ such that for any ${g \in {BS(1,n)}}$, ${x_{a^{k} \cdot g} = x_g}$ ($a^k$-periodicity).
  Call $w := (x_{a^j})_{j \in \Z}$ the biinfinite word based on level $\mathcal{L}_e$. $w$ is $k$-periodic by $a^k$-periodicity of the configuration $x$. But $w$ is also $nk$-periodic. Hence $(x_{b a^j})_{j \in \Z}$ is $k$-periodic. Indeed, by construction, when applying the correct substitution $\sigma_i$ to $x_{b a^j}$ and $x_{ba^{j+k}}$, one obtains the words $x_{a^{nj}} \dots x_{a^{nj+n-1}}$ and $x_{a^{nj+nk}} \dots x_{a^{nj+nk+n-1}}$ which are one and the same by $nk$-periodicity of $w$. Since there is only one preimage for a word by $\sigma_i$, $x_{b a^j} = x_{ba^{j+k}}$.
  By the same argument, one can show that for any integer $l > 0$, $(x_{b^l a^j})_{j \in \Z}$ must be $k$-periodic. However, these biinfinite sequences only use digits among $\{0,1,2\}$ so there is a finite number of such sequences. In particular, two of these sequences are the same. Since one is obtained from the other by applying the correct succession of $\sigma_{i}$'s, we get a periodic sequence that is a fixpoint of some $s = \sigma_{i_N}\circ\dots\circ\sigma_{i_1}$ for some $i_1, \dots, i_N\in \{0,\dots,n-1\}$. This contradicts \cref{lemma:fixpoint_aperiodic}.
\end{proof}

\begin{lemma}
  \label{lemma:weakap}
  There exists a weakly periodic configuration in $X^{\tau_\sigma}$ for $n \geq 3$.
\end{lemma}

\begin{proof}
  We define $w$ the unique fixpoint of $\sigma_1$ obtained thanks to \cref{lemma:fixpoint_unique}.

  Let $f(k) = \lfloor\frac{k}{n}\rfloor$ be the function that maps $k$ to the quotient in the euclidean division of $k$ by $n$ and $r(k)$ its remainder.
  We also define $F(k) = f(k+1)$ and $R(k) = r(k+1)$. This means that $nF(k) + R(k) = k+1$, $F(ln) = l + \lfloor\frac{1}{n}\rfloor = l$, but also $F(k+n) = \lfloor\frac{k+1+n}{n}\rfloor = F(k)+1$, and consequently $F^m(k+n^m) = F^m(k)+1$.

  \subsubsection*{\textbf{$X^{\tau_\sigma}$ is nonempty}}

  We define a configuration $x$ describing which tile $(c_g,i_g)$ (a tile being uniquely defined by such a couple) is assigned to $g$, i.e. $x_g = (c_g,i_g)$, using the quasi-normal form $g = b^{-k}a^lb^{m}$. 
  Then, we check that $x$ does verify the adjacency rules.
  Define $x \in {\tau_\sigma}^{BS(1,n)}$ by

  \[
  \begin{cases}
  x_{b^{-k}a^l} := (w_l,1)\\
  x_{b^{-k}a^lb^{m}} := (w_{F^{m}(l)},R \circ F^{m-1}(l))$ for $m > 0.
  \end{cases}
  \]

  Remember that \cref{lemmaBS1} states that any $g \in BS(1,n)$ can be written $b^{-k_1}a^{l_1}b^{m_1}$. Suppose it has a second form $b^{-k_2}a^{l_2}b^{m_2}$ with $k_2>k_1$ up to exchanging the notations (were they equal, it is easy to prove the two forms would be the same).
  Then $b^{k_1-k_2}a^{l_2}b^{m_2-m_1} = a^{l_1}$, that is, $b^{k_1-k_2}a^{l_2-l_1 n^{k_2-k_1}}b^{m_2-m_1} = e$. 
  This means that $k_2 - k_1 = m_2 - m_1 > 0$ and $l_2= l_1 n^{k_2-k_1}$.
  With that, we prove our $x$ is well-defined. $k_2>k_1$ causes $m_2>0$ in order to have $k_2 - k_1 = m_2 - m_1 > 0$. Consequently,

  \begin{align*}
  x_{b^{-k_2}a^{l_2}b^{m_2}} & = (w_{F^{m_2}(l_2)},R \circ F^{m_2-1}(l_2))\\
  & = (w_{F^{m_1+k_2-k_1}(l_1 n^{k_2-k_1})},R \circ F^{m_1+k_2-k_1-1}(l_1 n^{k_2-k_1}))\\
  & = (w_{F^{m_1}(l_1)},R \circ F^{m_1-1}(l_1))\\
  & = x_{b^{-k_1}a^{l_1}b^{m_1}}
  \end{align*}

  \noindent
  with a variation on the second to last line if $m_1=0$: we have $R \circ F^{k_2-k_1-1}(l_1 n^{k_2-k_1}) = R(l_1 n) = 1$.

  Now, we prove that $x \in X^{\tau_\sigma}$. Let $g = b^{-k}a^lb^{m}$.
    \begin{itemize}
        \item If $m>0$, we have
            \begin{align*}
            x_{ga}(\text{left}) & = x_{b^{-k}a^{l+n^{m}}b^{m}}(\text{left})\\
            & = R \circ F^{m-1}(l+n^{m})\\
            & = R(F^{m-1}(l)+n)\\
            & = R \circ F^{m-1}(l)\\
            & = x_g(\text{right}).
            \end{align*}
        \item If $m=0$, we have
            \begin{align*}
            x_{ga}(\text{left}) & = x_{b^{-k}a^{l+1}}(\text{left})\\
            & = 1\\
            & = x_g(\text{right}).
            \end{align*}
    \end{itemize}

  Let $j \in \{0,\dots,n-1\}$. We have
  \begin{align*}
  x_{ga^{-j}b}(\text{bottom}_{j+1}) & = x_{b^{-k}a^{l-jn^m}b^{m+1}}(\text{bottom}_{j+1}) & \text{($g=b^{-k}a^lb^m$)}\\
  & = \sigma_{R \circ F^{m}(l-jn^m)}(w_{F^{m+1}(l-jn^m)})_{j} & \text{(by definition of $x$)}\\
  & = \sigma_{R \circ F^{m}(l-jn^m)}(w)_{n F^{m+1}(l-jn^m)+j} & (\cref{lemma:exponentindex})\\
  & = \sigma(w)_{n F^{m+1}(l-jn^m)+j+R \circ F^{m}(l-jn^m)} & (\cref{lemma:remainder})\\
  & = \sigma(w)_{F^{m}(l-jn^m)+j+1} & \text {(by definition of } F \text{ and } R \text{)}\\
  & = \sigma_1(w)_{F^{m}(l-jn^m)+j} & (\cref{lemma:remainder})\\
  & = w_{F^{m}(l-jn^m)+j} & \text{(since $w$ is a fixpoint of $\sigma_1$)}\\
  & = w_{F^{m}(l)} & (F^{m}(l-jn^m) = F^m(l)-j)\\
  & = x_g(\text{top})
  \end{align*}

  Consequently, $x$ describes a valid configuration of $X^{\tau_\sigma}$: all adjacency conditions are verified.

    \subsubsection*{\textbf{$x$ is $b$-periodic}}

  With the definition of $x$, it is easy to check that for any $g \in BS(1,n)$, $x_{bg} = x_g$. Hence it is a weakly periodic configuration.
\end{proof}

We can now obtain our second main theorem:

\begin{theorem}
  \label{th:main}
  The tileset $\tau_\sigma$ forms a weakly aperiodic but not strongly aperiodic SFT on $BS(1,n), n \in \N, n \geq 2$.
\end{theorem}

\begin{proof}
  First, in the $n \geq 3$ case, there is a weakly periodic configuration in $X^{\tau_\sigma}$, see \cref{lemma:weakap}. Hence it is not a strongly aperiodic SFT.

  In the $n = 2$ case, we define $u$ and $v$ the two fixpoints of ${\sigma_1}^2$ (\cref{lemma:fixpoint_unique} again) and remark that $v = \sigma_1(u)$ and $u = \sigma_1(v)$. We define a configuration $x \in {\tau_\sigma}^{BS(1,n)}$ by:
  \[ x_{b^{-k}a^l}:=
  \begin{cases}
    (u_l,1)$ if $k+m \equiv 0$ mod $2\\
    (v_l,1)$ if $k+m \equiv 1$ mod $2
  \end{cases} \]
  \[ x_{b^{-k}a^lb^{m}}:=
  \begin{cases}
  (u_{F^{m}(l)},R \circ F^{m-1}(l))$ for $m > 0$ if $k+m \equiv 0$ mod $2\\
  (v_{F^{m}(l)},R \circ F^{m-1}(l))$ for $m > 0$ if $k+m \equiv 1$ mod $2
  \end{cases}\]
  and we use the same notations as in the proof of \cref{lemma:weakap}. The reasoning is also the same, except instead of using $w$ an alternation appears between $u$ and $v$ in all the equations. As a consequence, the configuration is $b^2$-periodic instead of $b$. Once again, $X^{\tau_\sigma}$ is consequently not strongly aperiodic.

  Now, using \cref{prop:period}, and since all powers of $a$
  are of infinite order in ${BS(1,n)}$, we get that for any valid configuration $x$ of $\tau_\sigma$, $|Orb_{BS(1,n)}(x)| = +\infty$, for any $n \geq 2$. Hence no configuration of $\tau_\sigma$ is strongly periodic, and so the SFT is weakly aperiodic.
\end{proof}

\section{A strongly aperiodic SFT on \texorpdfstring{$BS(n,n)$}{BS(n,n)}}
\label{sec:strongnn}

This section is a mere assembly of known results, that we think are worth gathering in the context of the current paper. 
It uses a theorem from \cite{computable} seen as an extension of the construction presented in \cite{Kari2008}. The idea behind that theorem is that $G \times \Z$ admits a strongly aperiodic SFT as soon as $G$ can encode piecewise affine functions. This is reflected by the $PA^\prime$ condition described in \cite{computable} and restated below.

\begin{definition}
  Let $k \in \N$. Let $\mathcal{F} = \{f_i\colon P_i \rightarrow P_i^{\prime} \mid i \in \{0,\dots,k\}\}$ be a finite set of piecewise affine rational homeomorphisms, where each $P_i$ and $P_i^{\prime}$ is a finite union of bounded rational polytopes of $\mathbb{R}^n$. Let $D = \bigcap_{i=1}^k P_i \cap \bigcap_{i=1}^k P_i^{\prime}$ be the common domain of all functions of $\mathcal{F}$ and their inverses.
  
  Let $S_{\mathcal{F}}$ be the closure of the set $\{f_i, {f_i}^{-1} \mid i \in \{1,\dots,k\} \}$ under composition. We define $G_{\mathcal{F}}$, the group $\{f|_D \mid f \in S_{\mathcal{F}}\}$.

  A finitely generated group $G$ is $PA^\prime$-recognizable if there exists a finite set $\mathcal{F}$ of piecewise affine rational homeomorphisms such that:

    (A) $G \cong G_{\mathcal{F}}$;

    (B) $\forall t \in D, \forall g \in \mathcal{F}, [\forall f \in \mathcal{F}, g(f (t)) = f(t)] \Rightarrow g = Id$.
\end{definition}

\begin{theorem}[\hspace{1sp}\cite{computable}, Th. 7]
  \label{PA}
  If $G$ is an infinite finitely generated $PA^{\prime}$-recognizable group, then $\Z \times G$ admits a strongly aperiodic SFT.
\end{theorem}

We need two additional propositions to obtain the desired result on $BS(n,n)$:

\begin{prop}[\hspace{1sp}\cite{commensurable}, Prop. 9 \& 10]
  \label{finiteindexap}
  If $G$ is a finitely generated group and $H$ is a finitely generated subgroup of $G$ of finite index, then we have the following:
  \begin{center}
    $H$ admits a weakly aperiodic SFT $\Leftrightarrow$ $G$ admits a weakly aperiodic SFT

    $H$ admits a strongly aperiodic SFT $\Leftrightarrow$ $G$ admits a strongly aperiodic SFT.
  \end{center}
\end{prop}

The following proposition is known, but we include a self-contained proof.

\begin{prop}
  \label{prop:BSnn}
  $BS(n,n)$ admits $\Z \times \mathbb{F}_n$ as a subgroup of finite index, where $\mathbb{F}_n$ is the free group of order $n$.
\end{prop}

\begin{proof}
  Let $H$ be the subgroup of $BS(n,n)$ generated by $\{a^n\} \cup \{ a^i b a^{-i} \mid i \in \{0,\dots,n-1\}\}$.
  First, $H$ is normal in $BS(n,n)$. We prove that $aHa^{-1} \subseteq H$ by verifying it on its generators: the only verification needed is $a a^{n-1} b a^{-(n-1)} a^{-1} = a^n b a^{-n} = b a^n a^{-n} = b$. Similarly, $a^{-1}Ha \subseteq H$; and finally, $bHb^{-1} \subseteq H$ (same for $b^{-1}$) since $b \in H$.
  Second, $H$ is isomorphic to $\Z \times \mathbb{F}_n$ through the following isomorphism (denoting $g_0, \dots, g_{n-1}$ the generators of $\mathbb{F}_n$ and $\epsilon$ its identity):
  \begin{alignat*}{4}
    \phi\colon & \Z \times \mathbb{F}_n & \longrightarrow & H\\
    & (1,\epsilon) & \longmapsto & a^n\\
    & (0,g_i) & \longmapsto & a^i b a^{-i}
  \end{alignat*}
  It is a morphism by construction, which is correctly defined since the only basic relation of $\Z \times \mathbb{F}_n$, that is $(1,\epsilon)\cdot(0,g_i) = (0,g_i)\cdot(1,\epsilon)$, is preserved in $H$: $a^n a^i b a^{-i} = a^i a^n b a^{-i} = a^i b a^n a^{-i} = a^i b a^{-i} a^n$.
  Said morphism is surjective, because $H$ is generated by $a^n$ and $\{ a^i b a^{-i} \mid i \in \{0,\dots,n-1\}\}$.
  Finally, it is also injective: let $g = (k, w) \in \Z \times \mathbb{F}_n$, with $w = (g_{i_1})^{e_1} \dots (g_{i_N})^{e_N}$ where the $e_l$ are in $\{-1,+1\}$.
  \[
  \phi(g) = e \Leftrightarrow a^{nk} a^{i_1} b^{e_1} a^{-i_1} a^{i_2} \dots a^{-i_{N-1}} a^{i_N} b^{e_N} a^{-i_N} = e
  \]
  This form is a canonical form in $H$: any word in $H$ can be uniquely written as such. Indeed, any word in $H$ is a succession of generators of it, $a^{i_k}b^{e_k}a^{-i_k}$ and $a^n$. But $a^n$ commutes with all the other generators due to the relation of $BS(n,n)$, so such a form is always attainable. To prove it is unique, it is enough to prove it for $e$: suppose we have some $a^{nk} a^{i_1} b^{e_1} a^{-i_1} a^{i_2} \dots a^{-i_{N-1}} a^{i_N} b^{e_N} a^{-i_N} = e$. First, realize that no relation in $BS(n,n)$ allows to reduce the total power of $a$ in a word, causing $k=0$ necessarily. Then, consider the resulting word $a^{i_1} b^{e_1} a^{i_2-i_1} \dots a^{i_N-i_{N-1}} b^{e_N} a^{-i_N}$ in $BS(n,n)$: it cannot be reduced in $BS(n,n)$ since all powers of $a$ between two $b$'s are of absolute value smaller than $n$.
  
  As a consequence, the previous equality is true only when $k=0$ and $w = \epsilon$. Hence the injectivity of the map.
  
  Moreover, any element of $BS(n,n)$ can be written in a form that much resembles the one mentioned above:
  \[
  a^p a^{nk} a^{i_1} b^{e_1} a^{-i_1} \dots a^{i_N} b^{e_N} a^{-i_N}
  \]
  with $p \in \{0,\dots,n-1\}$. To do so, first move all $a^n$'s in the rightmost power of $a$ in the word, to the leftmost part of the word. Ensure that $-i_N$, the remaining power, is in $\{-n+1, \dots, -1, 0\}$. Then force $a^{i_N}$ to appear on the left of the $b$ itself to the left of $a^{-i_N}$, and call $-i_{N-1}$ the remaining power of $a$ (it is in $\{-(n-1), \dots, -1, 0\}$ up to moving another $a^n$ to the leftmost part of the word) before another $b$ to the left. Repeat this operation until there is no $b$ to the left of the power of $a$ you consider, and split this final $a^K$ into $a^p a^{nk} a^{i_1}$.
  
  As a consequence, $BS(n,n) / H = \{\overline{e}, \overline{a}, \dots, \overline{a^{n-1}}\} \cong \Z / n\Z$.
  Hence $H$ is of finite index in $BS(n,n)$.
\end{proof}

\begin{theorem}
  \label{esnay}
  For every $n \geq 2$, $BS(n,n)$ admits a strongly aperiodic SFT.
\end{theorem}

\begin{proof}
  First, finitely generated subgroups of compact groups of matrices on integers are $PA^{\prime}$-recognizable (see \cite{computable}, Proposition 5.12). $\mathbb{F}_2$, the free group of order $2$, is isomorphic to a subgroup of $SL_2(\Z)$ (see \cite[Lemma 2.3.2]{cellulaut}), hence it is $PA^{\prime}$-recognizable.
  It is also known (see \cite[Corollary D.5.3]{cellulaut}) that $\mathbb{F}_n$ is a subgroup of $\mathbb{F}_2$; so it is isomorphic to a subgroup of $SL_2(\Z)$ and $PA^\prime$-recognizable too.
  Therefore by \cref{PA} $\Z \times \mathbb{F}_n$ admits a strongly aperiodic SFT.
  Using \cref{finiteindexap} and \cref{prop:BSnn}, we obtain that $BS(n,n)$ admits a strongly aperiodic SFT.
\end{proof}

\section*{Conclusion}

Baumlag-Solitar groups $BS(m,n)$ are residually finite if and only if $|m| = 1$, $|m| = 1$ or $|m| = |n|$ (\cref{prop:residually}). Gathering results from \cref{th:strong}, \cref{th:main} and \cref{esnay}, and considering that $BS(-m,-n) \cong BS(-m,n) \cong BS(m,-n) \cong BS(m,n)$, we obtain the following:
\begin{theorem}
    Residually finite Baumslag-Solitar groups $BS(m,n)$ with $|m| \geq 2$ or $|n| \geq 2$ admit both stronly and weakly-not-strongly aperiodic SFTs.
\end{theorem}

For non-residually finite Baumslag-Solitar groups, the existence of strongly aperiodic STF is still an open question.

In \cref{sec:weak1n}, we showed how to encode a particular set of substitutions into a tiling of $BS(1,n)$. 
An interesting question related to combinatorics on words would be to characterize the sets of substitutions that can be encoded using our technique.
It is clear that the nature of $BS(1,n)$ impose some conditions on these substitutions, and it would be of independent interest to obtain a self-contained condition on the substitutions and study their properties.

\section*{Acknowledgments}

The authors would like to thank Nathalie Aubrun for the interest she sparked about tilesets on the Baumslag-Solitar groups and the help she provided to understand her joint work with Jarkko Kari.
They also thank Silv\`ere Gangloff for pointing them to \cite{DGG2014} that provided the missing piece to prove \cref{lemmarepresent}; Jarkko Kari for his questions, that led to the writing of \cref{sec:weak1n}; and Pierre Guillon for his many remarks and useful suggestions that made the paper much more readable, even if it delayed a bit its publication.

The first author would like to thank Michael Schraudner and his PhD students \'Alvaro Bustos and Hugo Maturana, under whose supervision \cref{sec:strongnn} was conceived as part of an internship, for their numerous insights in the proof of said theorem.

\smallskip

This publication was made possible through the support of the ID\#61466 grant from the John Templeton Foundation, as part of the ``The Quantum Information Structure of Spacetime (QISS)'' Project (\url{qiss.fr}). The opinions expressed in this publication are those of the author(s) and do not necessarily reflect the views of the John Templeton Foundation

\bibliographystyle{abbrv}
\bibliography{biblio}

\end{document}